\documentclass[11pt]{amsart}

\usepackage{amsmath,amssymb,latexsym,soul,cite,mathrsfs}
\usepackage{mathtools}

\usepackage{color,enumitem,graphicx}
\usepackage[colorlinks=true,urlcolor=blue,
citecolor=red,linkcolor=blue,linktocpage,pdfpagelabels,
bookmarksnumbered,bookmarksopen]{hyperref}
\usepackage[english]{babel}

\usepackage[left=2.9cm,right=2.9cm,top=2.8cm,bottom=2.8cm]{geometry}
\usepackage[hyperpageref]{backref}

\usepackage[colorinlistoftodos]{todonotes}
\makeatletter
\providecommand\@dotsep{5}
\def\listtodoname{List of Todos}
\def\listoftodos{\@starttoc{tdo}\listtodoname}
\makeatother

\numberwithin{equation}{section}

\newtheorem{theorem}{Theorem}[section]
\newtheorem{proposition}[theorem]{Proposition}
\newtheorem{lemma}[theorem]{Lemma}

\newtheorem{remark}[theorem]{Remark}
\newtheorem{definition}[theorem]{Definition}

\newcommand\R{\mathbb R}
\newcommand\N{\mathbb N}

\newcommand\dist{\operatorname{dist}}

\begin{document}
	\title [Connecting Orbits in Quasilinear Conservative Systems]{Prescribed--Energy Connecting Orbits for Quasilinear Conservative Systems}
	\author{Renan J. S. Isneri and Piero Montecchiari}
	
	\address[Renan J. S. Isneri]
	{\newline\indent Unidade Acad\^emica de Matem\'atica
		\newline\indent 
		Universidade Federal de Campina Grande,
		\newline\indent
		58429-970, Campina Grande - PB - Brazil}
	\email{\href{renanisneri@mat.ufcg.edu.br}{renanisneri@mat.ufcg.edu.br}}
	
	\address[Piero Montecchiari]
	{\newline\indent Dipartimento di Ingegneria Civile, Edile e Architettura,
		\newline\indent 
		Universit\`a Politecnica delle Marche,
		\newline\indent
		Via Brecce Bianche, I?60131 Ancona, Italy}
	\email{\href{p.montecchiari@staff.univpm.it}{p.montecchiari@staff.univpm.it}}
	
\begin{abstract}
	We consider quasilinear conservative systems
	\[
	(\phi(|\dot q|)\dot q)'=\nabla V(q), \qquad q\in\R^{N},
	\]
	with $\Phi$-growth kinetic term and potential $V\in C^{1}(\R^{N};\R)$. Assuming that for some $c\in\R$ the sublevel set $\{V\le c\}$ splits into two disjoint closed subsets $\mathcal V_c^{-}$ and $\mathcal V_c^{+}$, we prove the existence of trajectories $q_c$ with prescribed energy $-c$ connecting $\mathcal V_c^{-}$ and $\mathcal V_c^{+}$, obtained through an energy-constrained variational method. Although the construction yields weak solutions in an Orlicz-Sobolev setting, minimal $c$-connections are shown to be classical $C^2$ trajectories satisfying the strong energy identity $E_{q_c}\equiv -c$. The resulting entire trajectories include heteroclinic, homoclinic, and brake-type orbits. Applications to double-well, Duffing-type, and multiple pendulum systems are discussed.
\end{abstract}


\maketitle

\renewcommand{\contentsname}{Contents} 
\setcounter{tocdepth}{2} 


\section{Introduction}

In this paper, we study quasilinear systems of the form
\begin{equation}\label{S}
	\left(\phi\left(|\dot{q}(t)|\right)\dot{q}(t)\right)'=\nabla V(q(t))\quad\text{on}\quad\mathbb{R},
\end{equation}
where $V\in C^1(\mathbb{R}^N;\mathbb{R})$ is an autonomous potential with $N\geq 1$, and $\phi:(0,+\infty)\rightarrow (0,+\infty)$ is a $C^1$ function satisfying  
\begin{itemize}
	\item[\hypertarget{phi1}{($\phi_1$)}] $(\phi(s)s)'>0$ for all $s>0$;
	\item[\hypertarget{phi2}{($\phi_2$)}] there exist $l,m\in\mathbb{R}$ with $1< l\leq m$ such that 
	$$
	l-1\leq \dfrac{(\phi(s)s)'}{\phi(s)}\leq m-1,\quad \forall\, s>0.$$
\end{itemize}	

The conditions \hyperlink{phi1}{$(\phi_1)$} and \hyperlink{phi2}{$(\phi_2)$} allow us to study system \eqref{S} within a variational framework in Orlicz Sobolev spaces. Central to this approach is the function $\Phi: \mathbb{R} \to [0,+\infty)$ defined by
\[
\Phi(s) = \int_0^{|s|} \phi(r)\, r \, dr.
\]
This function is convex, even, nonnegative, and vanishes only at zero. Moreover, under \hyperlink{phi1}{$(\phi_1)$} and \hyperlink{phi2}{$(\phi_2)$}, it satisfies
\[
\lim_{s \to 0} \frac{\Phi(s)}{s} = 0, \qquad \lim_{s \to +\infty} \frac{\Phi(s)}{s} = +\infty,
\]
so that $\Phi$ is an $N$-function, a class of functions widely studied in the framework of Orlicz spaces (see \cite{Adams, RaoRen} and Section~ \ref{sec:preliminaries}). 

Systems of the form \eqref{S} arise at a structural level in models where the constitutive response depends nonlinearly on the velocity or strain rate. In such contexts, the function $\phi$ encodes the rheological law of the medium and may exhibit degeneracy at low velocities ($\phi(0)=0$), singular behavior near rest ($\lim_{s\to 0^+}\phi(s)=+\infty$), or superlinear growth at large strains ($\lim_{s\to+\infty}\phi(s)/s=+\infty$). These features occur, for instance, in models of non--Newtonian fluids and rate--dependent solids; see, e.g., \cite{Bird, FuchsSeregin, Papanastasiou, Gurtin}. The structural assumptions \hyperlink{phi1}{$(\phi_1)$}--\hyperlink{phi2}{$(\phi_2)$}
encompass these degenerate, singular, and superlinear regimes.

Quasilinear equations involving $\Phi$--Laplacian type operators have been extensively studied in the elliptic setting, both in bounded and in unbounded domains. Variational methods in Orlicz--Sobolev spaces provide the natural functional framework for these problems (see, e.g., the classical works 
\cite{AlvesFigueiredoSantos2014,Azzollini,BadialeCitti,BonannoBisciRadulescu2011,FukagaiItoNarukawa1,FukagaiItoNarukawa2, MihailescuRadulescu2007, MihailescuRadulescu2008}). More recent contributions in this direction include 
\cite{AlvesSilva2025,BahrouniMissaouiRadulescu2025,RadulescuSantosTavares2023,SantosPontesSoares2023}. Problems related to transition-type solutions in the $\Phi$--Laplacian setting have been investigated in \cite{AlvesIsneriMontecchiari1,AlvesIsneriMontecchiari2,AlvesIsneriMontecchiari3}.

In this paper we adopt a different perspective, viewing \eqref{S} as a Lagrangian system and focusing on searching for global connecting solutions on the whole line. System \eqref{S} is formally the Euler--Lagrange system 
associated with the Lagrangian
\[
L(q,\dot q)=\Phi(|\dot q|)+V(q).
\]
Consequently, it  can be viewed as a conservative autonomous system, 
with conjugate momentum
$p=\partial_{\dot q}L=\phi(|\dot q|)\dot q$.
In this perspective, the quantity
\[
E_q=p\cdot \dot q-L=G(|\dot q|)-V(q),
\]
with $G(s)=\phi(s)s^2-\Phi(s)$, plays the role of a mechanical energy associated with the system. 

Building upon previous variational results for quasilinear equations 
\cite{AlvesIsneriMontecchiari1,AlvesIsneriMontecchiari2,AlvesIsneriMontecchiari3}, 
and adopting the prescribed--energy variational framework 
introduced for semilinear elliptic problems in 
\cite{alessio2013stationary, alessio2016brake, [AlMbrake]} 
and later adapted to Newtonian systems in \cite{AMZ}, 
we investigate the multiplicity of connecting orbits 
at prescribed energy for system \eqref{S}.

For $c\in\mathbb{R}$ we consider the sublevel set
\[
\mathcal V_c=\{x\in\mathbb{R}^{N}:V(x)\le c\}.
\] 
We assume that for some $c$ the set $\mathcal V_c$ admits a decomposition
into two nonempty closed well separated sets:
\begin{itemize}
	\item[\hypertarget{V}{($\mathcal{V}_c$)}]  There exists $c \in \mathbb{R}$ and $\emptyset\neq\mathcal{V}_c^-\subset \mathcal{V}_c$,  $\emptyset\neq\mathcal{V}_c^+\subset \mathcal{V}_c$ such that
	$$
	\mathcal{V}_c = \mathcal{V}_c^- \cup \mathcal{V}_c^+ \quad \text{and} \quad \operatorname{dist}(\mathcal{V}_c^-, \mathcal{V}_c^+) > 0,
	$$
\end{itemize}
(here $\dist(A,B)$ denotes the usual Euclidean distance in $\mathbb R^N$).

Under assumption \hyperlink{V}{$(\mathcal V_c)$}, we consider the problem of finding a 
$c$--connection between $\mathcal V_c^-$ and $\mathcal V_c^+$.
More precisely, we look for an interval $I=(\alpha,\omega)\subset\mathbb R$, 
with $-\infty\le \alpha<\omega\le +\infty$, and a function 
$q\in L^\infty(I;\mathbb R^N)$ such that
\begin{enumerate}
\item[(i)] $q\in C^{1}(\overline I)\cap C^{2}(I)$ and solves \eqref{S} on $I$;
\item[(ii)] $V(q(t))>c$ and 
$E_q(t)=G(|\dot q(t)|)-V(q(t))=-c$ for all $t\in I$;
\item[(iii)]
$
\lim_{t\to\alpha^+}\operatorname{dist}(q(t),\mathcal V_c^-)=
\lim_{t\to\omega^-}\operatorname{dist}(q(t),\mathcal V_c^+)=0.$
\end{enumerate}
The existence of $c$--connections is obtained through a constrained minimization
of the action functional
\[
J_c(q)=\int_{\mathbb R}\big(\Phi(|\dot q(t)|)+V(q(t))-c\big)\,dt
\]
over the admissible class
\[
\Gamma_c=\Big\{q\in W^{1,\Phi}_{loc}(\R;\R^N):
V(q(t))\ge c \ \forall t,\ 
\liminf_{t\to-\infty}\dist(q(t),\mathcal V_c^-)=
\liminf_{t\to+\infty}\dist(q(t),\mathcal V_c^+)=0\Big\}.
\]
Here $W^{1,\Phi}_{loc}(\R;\R^N)$ denotes the local Orlicz--Sobolev space 
associated with the $N$--function $\Phi$ 
(see Section~\ref{sec:preliminaries}).
The constraint $V\ge c$ makes the integrand nonnegative, hence $J_c$ is bounded from below on $\Gamma_c$.

A minimizer $q\in\Gamma_c$ is defined globally on $\mathbb R$ and,
by admissibility, satisfies
\[
\liminf_{t\to -\infty}\operatorname{dist}(q(t),\mathcal V_c^-)=
\liminf_{t\to +\infty}\operatorname{dist}(q(t),\mathcal V_c^+)=0.
\]
To extract a $c$--connection, we introduce contact times
$-\infty\le \alpha<\omega\le+\infty$
corresponding to the last contact with $\mathcal V_c^-$
and the subsequent first contact with $\mathcal V_c^+$ (see definitions \eqref{Definition-alpha}, \eqref{Definition-omega}),
and consider the associated open interval $I=(\alpha,\omega)$.
On this interval the constraint is inactive, namely
\[
V(q(t))>c \quad \text{for } t\in I,
\]
and $q$ is a weak solution of \eqref{S}.

A crucial step in the present quasilinear setting consists in proving 
that $q$ satisfies a weak form of energy conservation, namely
\[
E_q(t)=G(|\dot q(t)|)-V(q(t))\equiv -c
\quad \text{for a.e. } t\in I.
\]
This identity is obtained by exploiting the minimality of $q$
with respect to suitable time rescalings (see Lemma~\ref{l12}).
In particular, $\dot q(t)\neq 0$ for a.e.\ $t\in I$.
The absence of degeneracy along $I$ then allows one to recover
the classical $C^2$--regularity of $q$ on $I$.

The existence of such a minimizer is obtained under the following 
mild coercivity condition on $J_c$:
\begin{itemize}
\item[\hypertarget{J}{($\mathcal J_c$)}] 
There exists $R>0$ such that
\[
\inf_{q\in\Gamma_c}J_c(q)
=
\inf\Big\{J_c(q):\, q\in\Gamma_c,\ 
\|q\|_{L^\infty(\mathbb R;\mathbb{R}^N)}\le R\Big\}.
\]
\end{itemize}
This abstract assumption reflects the geometry of the potential $V$ 
and ensures that the minimization of $J_c$ can be carried out within 
a uniformly bounded class of trajectories. 
In Section \ref{sec:applications}
we show that \hyperlink{J}{$(\mathcal J_c)$} is naturally satisfied 
for several classical potentials.

Under assumptions \hyperlink{V}{$(\mathcal V_c)$} and \hyperlink{J}{$(\mathcal J_c)$}, 
the above scheme yields a minimal $c$--connection, 
namely the restriction of a minimizer of $J_c$ on $\Gamma_c$.
\begin{theorem}\label{thm:main}
Assume \hyperlink{phi1}{$(\phi_1)$}--\hyperlink{phi2}{$(\phi_2)$}. 
Let $V\in C^1(\R^N;\R)$ and let $c\in\R$ be such that 
\hyperlink{V}{$(\mathcal V_c)$} and \hyperlink{J}{$(\mathcal J_c)$} hold. 
Then system \eqref{S} admits a minimal $c$--connection 
between $\mathcal V_c^-$ and $\mathcal V_c^+$.
\end{theorem}

While the variational strategy is inspired by \cite{AMZ}, 
the extension to the present quasilinear framework is far from straightforward. 
In the Newtonian case the kinetic term is quadratic, 
the Lagrangian is smooth and uniformly convex in $\dot q$, 
and minimizers of the constrained functional are automatically classical solutions. 
In particular, conservation of energy follows from standard arguments.

In contrast, under \hyperlink{phi1}{$(\phi_1)$}--\hyperlink{phi2}{$(\phi_2)$} the kinetic term may be 
nonhomogeneous and possibly degenerate or singular at $\dot q=0$. 
Consequently, conservation of energy and classical $C^{2}$ regularity 
for weak minimizers, which are essentially automatic in the semilinear setting, 
must here be recovered through a refined variational analysis.

Regularity for weak solutions of systems such as \eqref{S} is delicate.
In the elliptic setting, $C^{1,\alpha}$ regularity under $\phi$--structure assumptions
was established by Lieberman \cite{Lieberman} in the scalar case, while in the
vectorial case the seminal work of Uhlenbeck \cite{Uhlenbeck} provides
$C^{1,\alpha}$ regularity for homogeneous systems.
However, these results do not directly apply in our situation due to the
nonhomogeneous and vectorial nature of \eqref{S}.

In the quasilinear setting, the prescribed--energy variational scheme 
produces a nontrivial segment which is first obtained as a weak solution. 
From minimality we derive a weak conservation law, which implies that 
the trajectory lies in the region $\{V>c\}$ almost everywhere. 
In this region the equation becomes nondegenerate, which allows us to 
establish classical $C^{2}$ regularity together with the strong energy 
identity $E_q\equiv -c$.

In this way, conservation of energy and classical regularity emerge 
as consequences of the constrained variational structure itself. 
The method thus carries over to a general quasilinear framework and 
selects the prescribed energy level even in the presence of degeneracy 
or singular behavior. As a consequence, the variational construction, 
although performed in the Orlicz--Sobolev setting, produces classical 
$C^2$ trajectories satisfying the strong energy identity $E_q \equiv -c$. 
This makes it possible to classify the resulting connecting orbits 
according to the nature of the contact times with the level set $\{V=c\}$.

As shown in Section \ref{Subsec3.3}, $c$--connections admit natural global 
continuations and therefore generate entire weak solutions of \eqref{S} 
on $\mathbb R$ with constant energy $-c$.
More precisely, if a $c$--connection is defined on an interval 
$I=(\alpha,\omega)$, it can be extended across finite contact times 
by time reflection and, when both contact times are finite, 
by periodic continuation.

The resulting global solutions can be classified according to the nature 
of the contact times with the level set $\{V=c\}$. 
If both contact times are infinite, the solution remains strictly in the region 
$\{V>c\}$ and is asymptotic as $t\to\pm\infty$ to the distinct sets 
$\mathcal V_c^\pm$, yielding a heteroclinic-type solution. 
If exactly one contact time is finite, the corresponding entire solution is asymptotic 
to the same set in both time directions and is of homoclinic type. 
If both contact times are finite, the reflection procedure produces 
a periodic solution oscillating between $\mathcal V_c^-$ and $\mathcal V_c^+$, 
namely a brake-type orbit.

Moreover, if at least one contact time is infinite, then $c$ must be 
a critical value of $V$. In particular, for regular values of $c$ 
the construction produces only brake--type periodic solutions.\medskip

Finally, following the approach of \cite{AMZ}, Section \ref{sec:applications} 
illustrates the abstract variational framework through several classical 
models, including double--well, Duffing--type, and multiple pendulum--type systems.
These examples provide concrete situations in which the assumptions 
\hyperlink{V}{$(\mathcal V_c)$} and \hyperlink{V}{$(\mathcal J_c)$} can be verified explicitly.

In these cases the potential $V$ possesses isolated minima at level $c=0$, 
and there exists a threshold value $c^{*}>0$, determined by the geometry 
of the potential, such that \hyperlink{V}{$(\mathcal V_c)$} and \hyperlink{V}{$(\mathcal J_c)$} hold for all 
$0\le c<c^{*}$. Theorem~\ref{thm:main} therefore yields $c$--connections 
that extend to entire solutions with prescribed energy $-c$ for all 
such values of $c$.

When $c=0$, the solutions are heteroclinic or homoclinic orbits 
connecting minima of $V$. For $0<c<c^{*}$ the resulting trajectories 
may be heteroclinic, homoclinic, or of brake type. In particular, 
whenever $c$ is a regular value of $V$, the construction yields 
brake--type periodic solutions.

In the case of multiple pendulum--type systems, the geometry of the 
sublevel set $\mathcal V_c$ admits different admissible decompositions 
into well--separated subsets. This leads to multiplicity results, yielding geometrically 
distinct connecting orbits at the same prescribed energy level, 
as described in Proposition~\ref{R:3}.

%
%
\section{Preliminaries}\label{sec:preliminaries}

In this section we collect some basic consequences of assumptions \hyperlink{phi1}{$(\phi_1)$} and \hyperlink{phi2}{$(\phi_2)$} on the function $\phi$, and recall the corresponding properties of the associated $N$-function $\Phi$ and of the Orlicz and Orlicz--Sobolev spaces.
Although all the results stated below are classical and can be found, for instance,
in \cite{FukagaiItoNarukawa1,Adams,RaoRen,Krasnoselskii},
we briefly indicate the proofs of the main structural estimates
for the reader's convenience.

We recall that $\phi:(0,+\infty)\to(0,+\infty)$ is a $C^1$ function satisfying assumptions \hyperlink{phi1}{$(\phi_1)$}--\hyperlink{phi2}{$(\phi_2)$}, namely $(\phi(s)s)'>0$ for all $s>0$ and
\[
l-1 \le \frac{(\phi(s)s)'}{\phi(s)} \le m-1,
\quad \forall\, s>0,
\]
for some $1<l\le m$.

Define $\Phi(s)=\int_0^{|s|}\phi(r)r\,dr$ for $s\in\mathbb R$.
Then $\Phi$ is even, convex, $\Phi(0)=0$, $\Phi(s)>0$ for $s\neq0$, and
$\Phi'(s)=\phi(|s|)s$ for $s\neq0$.
From \hyperlink{phi1}{$(\phi_1)$}--\hyperlink{phi2}{$(\phi_2)$} one deduces
\[
\lim_{s\to0^+}\frac{\Phi(s)}{s}=0,\qquad
\lim_{s\to+\infty}\frac{\Phi(s)}{s}=+\infty,
\]
hence $\Phi$ is an $N$-function. Indeed, set 
\[a(s)=\phi(s)s\text{ for }s>0\text{ and }a(0):=\lim_{s\to0^+}\phi(s)s.\]
Note that $a(0)\in \mathbb R$ since, by \hyperlink{phi1}{$(\phi_1)$}, the function $a$ is nondecreasing on $(0,+\infty)$. Hence  $a\in C([0,+\infty))\cap C^1(0,+\infty)$ and, by \hyperlink{phi2}{$(\phi_2)$}, it satisfies $(l-1)a(t)\le a'(t)t$ for all $t>0$.
Integrating this on $[0,s]$ gives $l\int_0^s a(t)\,dt\le a(s)s$ for all $s>0$.
Since $l>1$, this forces $a(0)=0$ and thus $\Phi(s)/s\to0$ as $s\to0^+$.
Moreover, since $a'(t)/a(t)\ge (l-1)/t$ for $t>0$,
integrating on $[1,s]$ yields $a(s)\ge a(1)s^{\,l-1}$ for $s>1$,
which implies $\Phi(s)/s\to+\infty$ as $s\to+\infty$.

Under assumptions \hyperlink{phi1}{$(\phi_1)$}--\hyperlink{phi2}{$(\phi_2)$}, $\Phi$ satisfies the two-sided growth condition
\begin{equation}\label{eq:Phi-growth}
l \le \frac{\phi(s)s^2}{\Phi(s)} \le m,
\qquad \forall\, s>0.
\end{equation}
Indeed \hyperlink{phi2}{$(\phi_2)$} gives $(l-1)a(t) \le a'(t)t \le (m-1)a(t)$ for $t>0$. Then, integrating on $[0,s]$, we obtain $l\Phi(s) \le a(s)s \le m
\Phi(s)$ and \eqref{eq:Phi-growth} follows.

As shown in \cite[Lemma 2.1]{FukagaiItoNarukawa1}, this implies the homogeneity-type estimate
\begin{equation}\label{eq:omogeneity}
\min\{t^l,t^m\}\,\Phi(s) \le \Phi(st) \le \max\{t^l,t^m\}\,\Phi(s),
\qquad \forall\, s,t\ge0.
\end{equation}
Indeed  for any $s>0$ and $\tau>0$, $\frac{d}{d\tau}\Phi(s\tau)=a(s\tau)s$
and from \eqref{eq:Phi-growth},
$\frac{l}{\tau}\le \frac{d}{d\tau}\log\Phi(s\tau)\le\frac{m}{\tau}$.
Integrating with respect to $\tau$ from $1$ to $t$ (or $t$ to $1$) gives
$\log t^l\le\log\frac{\Phi(st)}{\Phi(s)}\le\log t^m$ for $t\ge1$,
$\log t^m\le\log\frac{\Phi(st)}{\Phi(s)}\le\log t^l$ for $0<t\le1$.
Exponentiating yields \eqref{eq:omogeneity}.

From \eqref{eq:omogeneity} with $t=2$ we obtain
$\Phi(2s)\le 2^m\Phi(s)$ for all $s\ge0$. Hence, $\Phi$ satisfies the $\Delta_2$-condition with $\lambda=2^{m}$, which is a standard structural assumption in the Orlicz space framework.

Some further elementary estimates will be useful later.

 First note that, since $\phi(s)s\to0$ as $s\to0^+$, there exists a constant $C_0>0$ such that $
\phi(s)s\le C_0$ for all $s\in[0,1]$.
Combined with \eqref{eq:Phi-growth} this yields
\begin{equation}\label{eq:phi-bound}
\phi(s)s \le m\,\Phi(s)+C_0,\qquad \forall\, s\ge0.
\end{equation}

Second, we have
\begin{equation}\label{eq:convexity}
\Phi(|a\xi_{1}+b\xi_{2}|)
\le
2^{m-1}(\Phi(|\xi_{1}|)+\Phi(|\xi_{2}|))
\end{equation}
for any $0<a,b\le 1$ and $\xi_{1},\xi_{2}\in\mathbb R^{N}$. Indeed note that the map $z\in\mathbb R^{N}\mapsto\Phi(|z|)$ is convex. Hence, for any $a,b>0$ and $\xi_{1},\xi_{2}\in\mathbb R^{N}$
\[
\Phi(|a\xi_{1}+b\xi_{2}|)
=
\Phi(|\tfrac12(2a\xi_{1})+\tfrac12(2b\xi_{2})|)
\le
\tfrac12\Phi(2a|\xi_{1}|)
+
\tfrac12\Phi(2b\,|\xi_{2}|).
\]
Since $0<a,b\le 1$, by \eqref{eq:omogeneity} we have
$
\Phi(2a|\xi_{1}|)\le 2^{m}\Phi(|\xi_{1}|),
$ and $\Phi(2b\,|\xi_{2}|)\le 2^{m}\Phi(|\xi_{2}|)$, \eqref{eq:convexity} follows.

We now briefly recall the main properties of the associated Orlicz and Orlicz--Sobolev spaces that will be used in the sequel.

Let $I\subset\mathbb{R}$ be an interval.
The Orlicz space associated with $\Phi$ is defined by
\[
L^\Phi(I;\mathbb{R}^N)=
\Bigl\{
q\in L^1_{\mathrm{loc}}(I;\mathbb{R}^N):
\int_I \Phi(|q|/\lambda)\,dt<\infty
\text{ for some }\lambda>0
\Bigr\},
\]
endowed with the Luxemburg norm
\[
\|q\|_{L^\Phi(I;\mathbb{R}^N)}
=
\inf\Bigl\{
\lambda>0:
\int_I \Phi(|q|/\lambda)\,dt\le1
\Bigr\}.
\]
Since $\Phi$ satisfies the $\Delta_2$-condition,
\[
q\in L^\Phi(I;\mathbb{R}^N)
\;\Longleftrightarrow\;
\int_I \Phi(|q|)\,dt<\infty,
\]
and $\|q\|_{L^\Phi(I;\mathbb{R}^N)}<\infty$ if and only if the integral is finite.

The corresponding Orlicz--Sobolev space is
\[
W^{1,\Phi}(I;\mathbb{R}^N) = \Big\{q \in L^\Phi(I;\mathbb{R}^N) : \dot q \in L^\Phi(I;\mathbb{R}^N)\Big\},
\]
with the norm
\[
\|q\|_{W^{1,\Phi}(I;\mathbb{R}^N)} = \| q \|_{L^\Phi(I;\mathbb{R}^N)}+\|\dot q \|_{L^\Phi(I;\mathbb{R}^N)}.
\]
The $\Delta_{2}$-condition ensures that both the Orlicz space $L^\Phi(I;\mathbb{R}^N)$ and the Orlicz--Sobolev space $W^{1,\Phi}(I;\mathbb{R}^N)$ are Banach, separable, and reflexive spaces.

From \eqref{eq:omogeneity} and the superlinear growth of $\Phi$ at $+\infty$, one deduces that there exists $C>0$ such that $\Phi(s)+C\ge |s|$. Hence on bounded intervals we have continuous embeddings
\[
L^\Phi(I;\mathbb{R}^N)\hookrightarrow L^1(I;\mathbb{R}^N),\quad
W^{1,\Phi}(I;\mathbb{R}^N)\hookrightarrow W^{1,1}(I;\mathbb{R}^N).
\]

We conclude this section by highlighting a structural property of the vector field associated with $\phi$, which will play a crucial role in the regularity analysis.

\begin{lemma}\label{lem:ginvC1}
Assume \hyperlink{phi1}{$(\phi_1)$} and \hyperlink{phi2}{$(\phi_2)$}. 
Define $g:\mathbb{R}^N\to\mathbb{R}^N$ by $g(x)=\phi(|x|)x$ for $x\neq 0$ and $g(0)=0$.
Then $g\in C(\mathbb{R}^N;\mathbb R^{N})\cap C^1(\mathbb{R}^N\setminus\{0\};\mathbb{R}^N)$,
$g(\mathbb{R}^N\setminus\{0\})=\mathbb{R}^N\setminus\{0\}$, and $g$ is locally invertible on
$\mathbb{R}^N\setminus\{0\}$. In particular, the inverse map
$g^{-1}:\mathbb{R}^N\setminus\{0\}\to\mathbb{R}^N\setminus\{0\}$ is of class $C^1$
on any compact subset of $\mathbb{R}^N\setminus\{0\}$.
\end{lemma}

\begin{proof}
Since $\phi(s)s\to0$ as $s\to0^+$, we have
$|g(x)|=\phi(|x|)|x|\to0$ as $x\to0$.
Hence $g$ extends continuously at $0$, and therefore
$g\in C(\mathbb{R}^N;\mathbb{R}^N)$.

To show that $g(\mathbb{R}^N\setminus\{0\})=\mathbb{R}^N\setminus\{0\}$, let $a(r)=\phi(r)r$ for $r>0$. By \hyperlink{phi1}{$(\phi_1)$}, $a'(r)=(\phi(r)r)'>0$,
hence $a$ is strictly increasing. Moreover, by \hyperlink{phi2}{$(\phi_2)$} one has
$\lim_{r\to0^+}a(r)=0$ and $\lim_{r\to+\infty}a(r)=+\infty$, so $a$ is a continuous
bijection $(0,+\infty)\to(0,+\infty)$. Writing $x=r\omega$ with $r>0$, $|\omega|=1$, we have
$g(x)=a(r)\omega$, which implies $g(\mathbb{R}^N\setminus\{0\})=\mathbb{R}^N\setminus\{0\}$.

For $x\neq0$ we have
\[
\frac{\partial g_i}{\partial x_j}(x)
=
\phi(|x|)\delta_{ij}
+
\phi'(|x|)\frac{x_i x_j}{|x|},
\]
and hence, in matrix form,
\[
Dg(x)=\phi(|x|)I+\phi'(|x|)\frac{x\otimes x}{|x|},
\qquad
x\otimes x:=(x_i x_j)_{1\le i,j\le N}.
\]
Then $g\in C^{1}(\mathbb R^{N}\setminus\{0\};\mathbb\R^{N})$. Moreover, using the assumptions \hyperlink{phi1}{$(\phi_1)$} and \hyperlink{phi2}{$(\phi_2)$} and the rank-one perturbation formula
$\det(A+uv^T)=\det(A)(1+v^TA^{-1}u)$ with
$A=\phi(|x|)I$ and $uv^T=\phi'(|x|)\frac{x\otimes x}{|x|}$, we obtain
\[
\det(Dg(x))
=
\phi(|x|)^{\,N-1}(\phi(|x|)+\phi'(|x|)|x|)
>0
\qquad \text{for all } x\neq 0.
\]
Therefore, $g$ is a local diffeomorphism on
$\mathbb{R}^N\setminus\{0\}$.
Since $g(\mathbb{R}^N\setminus\{0\})=\mathbb{R}^N\setminus\{0\}$,
its inverse $g^{-1}$ is of class $C^1$
on any compact subset of $\mathbb{R}^N\setminus\{0\}$.
\end{proof}

%
%
\section{Minimal 
$c$--connections and their global extensions
}\label{Section3}

In this section we develop the variational theory leading to the existence of solutions for the quasilinear system \eqref{S}, which connect the two sublevel sets $\mathcal V_c^-$ and $\mathcal V_c^+$ at a prescribed energy level $-c$. Such solutions will be referred to as $c$--connections.

The analysis is carried out within the abstract framework introduced in the previous sections. Throughout this section we assume that the structural conditions \hyperlink{phi1}{$(\phi_1)$}, \hyperlink{phi2}{$(\phi_2)$}, \hyperlink{V}{$(\mathcal V_c)$}, and \hyperlink{J}{$(\mathcal J_c)$} are satisfied. 

We also study qualitative properties of such trajectories. In particular we describe the structure of the contact times with the sublevel sets, prove that the solutions have constant energy $-c$, and that they are classical $C^2$ solutions of \eqref{S} away from the contact points.

%
%
\subsection{Minimization problem on $\Gamma_{c}$}
Given $c\in\mathbb{R}$ and an interval $I\subset\mathbb{R}$, we define the action functional
\[
J_{c,I}(q)=\int_{I}\bigl(\Phi(|\dot{q}(t)|)+V(q(t))-c\bigr)dt
\]
on the class
\[
\mathcal{X}_c=\bigl\{ q \in W^{1,\Phi}_{\text{loc}}(\mathbb{R};\mathbb{R}^N) : \inf_{t\in\mathbb{R}} V(q(t)) \ge c \bigr\}.
\]
When $I=\mathbb{R}$ we write $J_c$ instead of $J_{c,\mathbb{R}}$.

Functions in $W^{1,\Phi}_{\text{loc}}(\mathbb{R};\mathbb{R}^N)$ are absolutely continuous, due to the continuous embedding
\[
W^{1,\Phi}_{\text{loc}}(I;\mathbb{R}^N)\hookrightarrow W^{1,1}_{\text{loc}}(I;\mathbb{R}^N)
\]
(see Section~\ref{sec:preliminaries}). 
If $(q_n)\subset W^{1,\Phi}_{\text{loc}}(\mathbb{R};\mathbb{R}^N)$ and $q\in W^{1,\Phi}_{\text{loc}}(\mathbb{R};\mathbb{R}^N)$, we write 
$q_n\rightharpoonup q$ in $W^{1,\Phi}_{\text{loc}}(\mathbb{R};\mathbb{R}^N)$ to mean that $q_n$ converges weakly to $q$ in $W^{1,\Phi}(I;\mathbb{R}^N)$ for every compact interval $I\subset\mathbb{R}$.

The next lemma shows that $J_{c,I}$ is nonnegative on $\mathcal{X}_c$ and has useful semicontinuity properties.

%
%

\begin{lemma}\label{l1}
	For any $c\in\mathbb{R}$ and $I\subset\mathbb{R}$, we have $J_{c,I}(q)\geq 0$ for all $q\in\mathcal{X}_c$. 
	Moreover, if $(q_n)\subset\mathcal{X}_c$ with $q_n\rightharpoonup q$ in $W^{1, \Phi}_{\mathrm{loc}}(\mathbb{R};\mathbb{R}^N)$ for some $q\in W^{1, \Phi}_{\mathrm{loc}}(\mathbb{R};\mathbb{R}^N)$, then 
	$$
	J_{c,I}(q)\leq \liminf_{n\to\infty} J_{c,I}(q_n).
	$$
\end{lemma}
\begin{proof}
Since $\Phi(|\dot q(t)|)\ge 0$ and $V(q(t))\ge c$ for all $t\in\mathbb{R}$ 
for every $q\in\mathcal{X}_c$, we plainly have
	\[
	J_{c,I}(q)=\int_{I}\big(\Phi(|\dot q|)+V(q)-c\big)\,dt \;\geq\; 0, \qquad \forall\, q\in\mathcal{X}_c.
	\]
	Now let $(q_n)\subset\mathcal{X}_c$ and $q\in W^{1, \Phi}_{\mathrm{loc}}(\mathbb{R};\mathbb{R}^N)$ be such that $q_n\rightharpoonup q$ in $W^{1, \Phi}_{\mathrm{loc}}(\mathbb{R};\mathbb{R}^N)$.  
If $I\subset\mathbb{R}$ is a bounded interval, $q_n\rightharpoonup q$ in $W^{1, \Phi}(I;\mathbb{R}^N)$. Then, $(q_{n})$ is bounded in $W^{1,\Phi}(I;\mathbb{R}^N)$
and since $W^{1,\Phi}(I;\mathbb{R}^N)\hookrightarrow W^{1,1}(I;\mathbb{R}^N)$ continuously, the sequence is equicontinuous on $I$.  Hence $q_n\to q$ uniformly on $I$ and since $V$ is continuous 
\[
\int_{I}(V(q)-c)\,dt = \lim_{n\to\infty}\int_{I}(V(q_n)-c)\,dt.
\]
Moreover, since the functional 
$
v \mapsto \int_{I}\Phi(|\dot v|)\,dt
$
is convex and continuous on $W^{1,\Phi}(I;\mathbb{R}^N)$, it is weakly lower semicontinuous on $W^{1,\Phi}(I;\mathbb{R}^N)$. Hence
\[
\int_{I}\Phi(|\dot q|)\,dt \;\leq\; \liminf_{n\to\infty}\int_{I}\Phi(|\dot q_n|)\,dt.
\]
Then 
$J_{c,I}(q)\leq \liminf_{n\to\infty} J_{c,I}(q_n)$ for any bounded interval $I\subset\mathbb R$. 
If $I$ is not bounded, set $I_R:= I\cap(-R,R)$ for $R>0$. Since the integrand in $J_{c,I}$ is nonnegative, we have
\[
J_{c,I}(q)=\sup_{R>0}J_{c,I_R}(q), \qquad 
J_{c,I}(q_n)=\sup_{R>0}J_{c,I_R}(q_n).
\]
Therefore, taking the supremum over $R>0$ and using $\sup_R\liminf_n \le \liminf_n\sup_R$, we obtain
\[
J_{c,I}(q)=\sup_{R>0}J_{c,I_R}(q)\le \liminf_{n\to\infty}\sup_{R>0}J_{c,I_R}(q_n)=\liminf_{n\to\infty}J_{c,I}(q_n),
\]
and the lemma follows.
\end{proof}

We look for connecting solutions of \eqref{S} by minimizing $J_c$ on $\Gamma_c$, and we set
\[
m_c = \inf_{q\in\Gamma_c} J_c(q).
\]
To get compactness in the minimization process, it is useful to recall an important coercivity inequalities already used in \cite{AlvesIsneriMontecchiari1}. 

\begin{lemma}\label{l4}
Let $q\in\mathcal{X}_c$ and $(\sigma,\tau)\subset\mathbb{R}$. If there exists $d>0$ such that $V(q(t))-c\geq d$ for all $t\in(\sigma,\tau)$, then 
	\begin{equation*}
		\begin{split}
			J_{c,(\sigma,\tau)}(q)&\geq \dfrac{\tau-\sigma}{\xi_1(\tau-\sigma)}\Phi(|q(\tau)-q(\sigma)|)+d(\tau-\sigma)\\& \geq \lambda_d H\left(|q(\tau)-q(\sigma)|\right),
		\end{split}
	\end{equation*} 
	where $\lambda_d := \min\{ d^{(l-1)/l}, d^{(m-1)/m} \}$, $H(s)=\min\left\{\Phi(s)^{\frac{1}{l}},\Phi(s)^{\frac{1}{m}}\right\}$ and $\xi_1(s)=\max\{s^l, s^m\}$. 
\end{lemma}
\begin{proof}
The argument follows the approach in \cite[Lemma 3.2]{AlvesIsneriMontecchiari1}, and we only briefly outline it. Let $q \in \mathcal{X}_c$ satisfy $V(q(t)) - c \ge d$ for all $t \in (\sigma, \tau)$.  
By the monotonicity and convexity of $\Phi$ and Jensen's inequality
\[
\Phi\big(|\int_\sigma^\tau \dot{q}(t)\, dt|\big)\le \Phi\big(\int_\sigma^\tau |\dot{q}(t)|\, dt\big) \le \tfrac{1}{\tau-\sigma} \int_\sigma^\tau \Phi((\tau-\sigma)|\dot{q}(t)|)\, dt.
\]  
Using the estimate \eqref{eq:omogeneity}, denoting $\xi_1(\tau-\sigma)=\max\{(\tau-\sigma)^{l}, (\tau-\sigma)^{m}\}$, this yields that
\[
\Phi(|q(\tau)-q(\sigma)|) \le \tfrac{\xi_1(\tau-\sigma)}{\tau-\sigma} \int_\sigma^\tau \Phi(|\dot{q}(t)|)\, dt.
\]  
Hence
\[
J_{c,(\sigma,\tau)}(q) \ge \frac{\tau-\sigma}{\xi_1(\tau-\sigma)} \Phi(|q(\tau)-q(\sigma)|) + d(\tau-\sigma).
\]  
Since $\xi_1(\tau-\sigma) = (\tau-\sigma)^p$ for a $p \in \{l,m\}$ (depending on whether $\tau-\sigma \le 1$ or 
$\tau-\sigma \ge 1$), applying Young's inequality with conjugate exponents $p$ and $p/(p-1)$, we obtain
\[
J_{c,(\sigma,\tau)}(q) \ge d^{(p-1)/p} \, \Phi(|q(\tau)-q(\sigma)|)^{1/p}.
\]  
Taking the minimum of the right hand side with respect to  $p \in \{l,m\}$, 
we obtain
\[
J_{c,(\sigma,\tau)}(q) \ge \min\{ d^{(l-1)/l}, d^{(m-1)/m} \}\, \min\left\{\Phi(|q(\tau)-q(\sigma)|)^{\frac{1}{l}},\Phi(|q(\tau)-q(\sigma)|)^{\frac{1}{m}}\right\},
\] 
completing the proof.
\end{proof}

Lemma \ref{l4} provides the basic coercivity estimate needed for compactness. 
It shows that whenever the potential stays above level $c$ on an interval, the action grows both with the length of the interval and with the endpoint displacement. 
This dual control prevents trajectories of finite action from spreading indefinitely in time or in space.
%

\begin{remark}\label{remark1}
We denote by $\rho_0$ the distance between $\mathcal{V}_c^-$ and $\mathcal{V}_c^+$:  
\[
\rho_0 = \operatorname{dist}(\mathcal{V}_c^-, \mathcal{V}_c^+).
\] 
By \hyperlink{V}{$(\mathcal{V}_c)$}, we have $\rho_0>0$. Moreover, for given $r>0$ and $C>0$, the set
\[
K_{r,C} = \bigl\{x\in\mathbb{R}^N : |x|\leq C \text{ and } \operatorname{dist}(x,\mathcal{V}_c)\geq r \bigr\}
\]
is compact and, due to the continuity of $V$, there exists $d=d(r,C)>0$ such that 
\[
\inf_{x\in K_{r,C}} V(x) \geq c + d.
\]
\end{remark}

A first consequence of Lemma \ref{l4} is that the minimal level $m_{c}$  cannot vanish. 

\begin{lemma}\label{l3}
$m_c \in (0,+\infty)$.
\end{lemma}

\begin{proof}
It is immediate to see that $m_c < +\infty$. 

Indeed, by \hyperlink{V}{$(\mathcal{V}_c)$}, 
\begin{equation}\label{eq:segment}\text{there exist }x_- \in \mathcal{V}_c^-\text{ and }x_+ \in \mathcal{V}_c^+\text{ such that }V(tx_+ + (1-t)x_-) > c\text{ for all }t \in (0,1).\end{equation}  
Choose any pair $\xi^-\in\mathcal V_c^-$, $\xi^+\in\mathcal V_c^+$ and set $\gamma(s)=(1-s)\xi^-+s\xi^+$.
Let $s_-:=\sup\{s\in[0,1]:\gamma(s)\in\mathcal V_c^-\}$ and $s_+:=\inf\{s>s_-:\gamma(s)\in\mathcal V_c^+\}$. Since $\mathcal V_c^\pm$ are closed, $x_{-}:=\gamma(s_-)\in\mathcal V_c^-$ and $x_{+}:=\gamma(s_+)\in\mathcal V_c^+$, and by definition $\gamma(s)\notin\mathcal V_c=\mathcal V_c^-\cup\mathcal V_c^+$ for all $s\in(s_-,s_+)$. Setting $s=(1-t)s_{-}+ts_{+}$ we have $V(tx_++(1-t)x_-)>c$ for all $t\in (0,1)$ and \eqref{eq:segment} follows.

Then, by \eqref{eq:segment}, the piecewise linear function
\[
\overline{q}(t) = \begin{cases}
x_-, & t \le 0,\\
t x_+ + (1-t)x_-, & t \in (0,1),\\
x_+, & t \ge 1,
\end{cases}
\]
belongs to $\Gamma_c$, and hence
\[
m_c \le J_c(\overline{q}) \le \int_0^1 \bigl( \Phi(|\dot{\overline{q}}(t)|) + V(\overline{q}(t)) - c \bigr)\, dt < +\infty.
\] 
To show that $m_c>0$, consider a $q\in\Gamma_{c}$ with $\|q\|_{L^{\infty}(\R;\R^{N})}\leq R$. By definition of $\Gamma_{c}$
we have $\liminf_{t\to\pm\infty}\dist(q(t),\mathcal{V}_c^\pm)=0$
As $\mathcal{V}_c^\pm$ are closed and $q$ is continuous, the maps $t\in\R\to \dist(q(t),\mathcal{V}_c^\pm)$ are continuous. As $\liminf_{t\to+\infty}\dist(q(t),\mathcal{V}_c^{+})=0$ and $\dist(\mathcal{V}_c^{-},\mathcal{V}_c^{+})=\rho_{0}$, we recover $\limsup_{t\to +\infty}\dist(q(t),\mathcal{V}_c^-)\geq \rho_{0}$. Since
$\liminf_{t\to-\infty}\dist(q(t),\mathcal{V}_c^-)=0$, by continuity there exists an interval $(\sigma,\tau)\subset\R$ such that
\begin{equation}\label{eq:1}
	\dist(q(\sigma),\mathcal{V}_c^-)=\rho_{0}/4,\ \dist(q(\tau),\mathcal{V}_c^-)=\rho_{0}/2,\ \rho_{0}/4\leq \dist(q(t),\mathcal{V}_c^-)\leq \rho_{0}/2\ \forall t\in (\sigma,\tau).
\end{equation}
By \eqref{eq:1} and Remark \ref{remark1} we deduce that $\dist(q(t),\mathcal{V}_c)\geq\rho_{0}/4$ for any $t\in (\sigma,\tau)$. Hence, since $\|q\|_{L^{\infty}(\R;\R^{N})}\leq R$, we have that $q(t)\in K_{\rho_{0}/4,R}$ for any $t\in (\sigma,\tau)$. By Remark \ref{remark1}, there exists $d=d(\rho_{0}/4,R)>0$ such that
\begin{equation*}\label{eq:Vbound}
	V(q(t))\geq c+d,\quad \forall t\in (\sigma,\tau).
\end{equation*}
Since \eqref{eq:1} implies also $|q(\tau)-q(\sigma)|\geq \rho_{0}/4$, Lemma  \ref{l4} yields
\[
J_{c,(\sigma,\tau)}(q)\ge\lambda_{d}\,H(\rho_{0}/4),\quad\forall q\in\Gamma_{c}\text{ such that }\|q\|_{L^{\infty}(\R;\R^{N})}\leq R,
\]
and therefore, from \hyperlink{J}{($\mathcal{J}_c$)}, we conclude $m_{c}\geq \lambda_{d}\,H(\rho_{0}/4)$.
\end{proof} 

A second application of Lemma \ref{l4} shows that any finite--action trajectory in $\Gamma_c$ with bounded $L^{\infty}$ norm must approach the sets $\mathcal{V}_c^\pm$ continuously as $t\to\pm\infty$.

\begin{lemma}\label{l5}
	Let $q \in \Gamma_c$ with $J_c(q) < +\infty$ and $\|q\|_{L^{\infty}(\R;\R^{N})}\leq R$. Then
	$$
	\lim_{t \to -\infty} \mathrm{dist}(q(t), \mathcal{V}_c^-) = \lim_{t \to +\infty} \mathrm{dist}(q(t), \mathcal{V}_c^+) = 0.
	$$
\end{lemma}
\begin{proof}
Since $q\in\Gamma_{c}$ we have $\liminf_{t \to -\infty} \mathrm{dist}(q(t), \mathcal{V}_c^-)=0$. Assume, by contradiction, that $\limsup_{t \to -\infty} \mathrm{dist}(q(t), \mathcal{V}_c^-) > 2r$ for an $r\in \left(0,\rho_0/4\right)$.
The continuity of the map $t\in\mathbb{R}\mapsto\mathrm{dist}(q(t), \mathcal{V}_c^-)$ implies the existence of a sequence of disjoint intervals $(s_n,t_n)$, with $t_{n+1} < s_n < t_n$ and $t_n \to -\infty$, such that 
\[
	|q(t_n) - q(s_n)|\geq r\quad\text{and}\quad r \le \mathrm{dist}(q(t), \mathcal{V}_c^-) \le 2r,\quad \forall\,\, t \in (s_n, t_n)\,\text{ and }\, n\in\mathbb{N}.
\]
By the definition of $\rho_{0}$, this implies
$\mathrm{dist}(q(t), \mathcal{V}_c)\geq r$, so $q(t)\in K_{r,R}$ for any  $t\in(s_n, t_n)$ and $n\in\mathbb{N}$.
Then, by Remark \ref{remark1}, there exists $d=d(r,R)>0$ such that $V(q(t))\geq c + d$ for any  $t\in(s_n, t_n)$ and $n\in\mathbb{N}$. Applying Lemma \ref{l4}, we deduce
$
J_{c, (s_n, t_n)}(q) \ge \lambda_{d} H(r)$, for any $n \in \mathbb{N}$
and so,
$$
J_c(q) \ge \sum_{n=1}^{\infty} J_{c, (s_n, t_n)}(q)\ge \sum_{n=1}^{\infty}\lambda_{d} H(r) = +\infty,
$$
which contradicts the assumption $J_c(q) < +\infty$. Therefore, $\lim_{t \to -\infty} \mathrm{dist}(q(t), \mathcal{V}_c^-)=0$. Similarly, one obtains $\lim_{t \to +\infty} \mathrm{dist}(q(t), \mathcal{V}_c^+) = 0$, and the proof is complete.
\end{proof}

As a further consequence of Lemma \ref{l4}, we can describe the concentration behavior of trajectories in $\Gamma_c$ whose action is close to the minimal level $m_c$. In particular, such trajectories are forced to transition from a neighborhood of $\mathcal{V}_c^-$ to a neighborhood of $\mathcal{V}_c^+$ within a controlled time interval.

\begin{lemma}\label{l6} There exist  $\nu_{0}>0$ and $L_{0}>0$ such that, 
if $q \in \Gamma_c$ satisfies 
\begin{equation}\label{eq:condizioniq}
\mathrm{dist}(q(0), \mathcal{V}_c) \ge \tfrac{\rho_0}{8},\quad 
\|q\|_{L^\infty(\mathbb{R}; \mathbb{R}^N)} \le R 
\quad\text{and}\quad 
J_c(q) \le m_c + \nu_0,
\end{equation} 
then 
\[
\dist(q(t), \mathcal{V}_c^+) < \tfrac{\rho_0}{8}\text{ for all } t > L_{0},
\text{ and }
\dist(q(t), \mathcal{V}_c^-) < \tfrac{\rho_0}{8}\text{ for all } t < -L_{0}.
\] 
\end{lemma}
\begin{proof}  
We begin by fixing the constants $\nu_0$ and $L_0$. 
By Remark~\ref{remark1}, there exists $d_{\rho_0} > 0$ such that
\begin{equation}\label{eq:d0}
V(x) \geq c + d_{\rho_0} \quad \text{for all } x \text{ with } |x| \leq R + \tfrac{\rho_0}{8} 
\text{ and } \dist(x, \mathcal V_c) \geq \tfrac{\rho_0}{16}.
\end{equation}
For $r \in (0, \tfrac{\rho_0}{16})$, define
\[
\mu_r = \max\Bigl\{V(x) - c : |x| \le R + \tfrac{\rho_0}{16},\ 
\dist(x, \mathcal{V}_c) \le r\Bigr\},
\]
and note that $\mu_r \to 0$ as $r \to 0^{+}$. 
Since $\Phi(r) \to 0$ as $r \to 0^{+}$, we can fix 
$r_0 \in (0, \tfrac{\rho_0}{16})$ such that
\begin{equation}\label{eq:nu0}
\nu_0 := \Phi(r_0) + \mu_{r_0} < \tfrac{1}{2}\lambda_{d_{\rho_0}} 
H(\tfrac{\rho_0}{16}),
\end{equation}
where $\lambda_{d_{\rho_0}}$ and $H$ are defined in Lemma \ref{l4}. Applying Remark~\ref{remark1} again with this $r_0$, we obtain $d_{r_0} > 0$ satisfying
\begin{equation}\label{eq:dr0}
V(x) \geq c + d_{r_0} \quad \text{for all } x \text{ with } |x| \leq R + \tfrac{\rho_0}{8} 
\text{ and } \dist(x, \mathcal V_c) \geq r_0.
\end{equation}
Finally, we set
\begin{equation*}\label{eq:L0}
L_0 := \tfrac{m_c + \nu_{0}}{d_{r_0}}.
\end{equation*}
With $\nu_0$ and $L_0$ now fixed, let $q \in \Gamma_c$ satisfy \eqref{eq:condizioniq}. Define
\[
\sigma_{0}= \sup\{ t \in \mathbb{R} : \text{dist}(q(t),\mathcal{V}_c^-) \leq r_{0} \} \text{ and }  \tau_{0} = \inf\{ t > \sigma_{0} : \text{dist}(q(t),\mathcal{V}_c^+) \leq r_{0}\}.\]
By Lemma~\ref{l5} the sets in the definitions of $\sigma_0$ and $\tau_0$ are nonempty, hence  $\sigma_0 \leq \tau_0 \in \mathbb{R}$. 
By continuity,
\[
\dist(q(\sigma_0),\mathcal{V}_c^-) = r_0 < \tfrac{\rho_0}{16},
\quad 
\dist(q(\tau_0),\mathcal{V}_c^+) = r_0 < \tfrac{\rho_0}{16},
\]
and since $\dist(\mathcal V_c^{-},\mathcal V_c^{+})=\rho_0$, we have $\sigma_0<\tau_0$. 
Moreover, by definition,
\begin{equation}\label{5}
\dist(q(t),\mathcal{V}_c) > r_0 \quad \text{for all } t \in (\sigma_0, \tau_0).
\end{equation}
We claim that
\begin{equation}\label{6}
	\text{dist}(q(t),\mathcal{V}_c^-) < \tfrac{\rho_0}{8}, \quad \forall\,\, t \leq \sigma_{0}.
\end{equation}
Suppose, by contradiction, that there exists $ t^* < \sigma_{0} $ such that
$$
\text{dist}(q(t^*),\mathcal{V}_c^-) \geq \tfrac{\rho_0}{8}.
$$
Since $\text{dist}(q(\sigma_{0}),\mathcal{V}_c^-) \leq \tfrac{\rho_0}{16} $, continuity yields an interval $ (\sigma, \tau) \subset (t^{*}, \sigma_0) $ where
$$
|q(\tau) - q(\sigma)| \ge \tfrac{\rho_0}{16} \text{ and }\tfrac{\rho_0}{16} \leq \text{dist}(q(t),\mathcal{V}_c^-) \leq \tfrac{\rho_0}{8}, \text{ for all }t \in (\sigma, \tau) .
$$
Hence, 
$
\operatorname{dist}(q(t),\mathcal{V}_c)\geq\frac{\rho_0}{16}$, for all $t\in (\sigma, \tau),
$
and, as $ \|q\|_{L^\infty(\mathbb{R};\mathbb{R}^N)} \leq R $, by \eqref{eq:d0} and Lemma \ref{l4} \begin{equation}\label{7}
	J_{c,(\sigma,\tau)}(q) \geq \lambda_{d_{\rho_{0}}} H(|q(\tau)-q(\sigma)|)\geq \lambda_{d_{\rho_{0}}} H(\tfrac{\rho_0}{16}).
\end{equation}
On the other hand, because $\text{dist}(q(\sigma_0),\mathcal{V}_c^-) =r_{0}<\rho_0/16<\rho_0=\dist(\mathcal V_c^-,\mathcal V_c^+)$, arguing as in the proof of \eqref{eq:segment}, we derive that there is $\xi_{\sigma_0} \in \mathcal{V}_c^- $ such that 
\begin{equation}\label{8}
	V(\xi_{\sigma_0})=c,\quad |q(\sigma_0) - \xi_{\sigma_0}| =r_{0}\quad\text{and}\quad V((1 - s)q(\sigma_0) + s\xi_{\sigma_0}) > c\quad \forall\,\, s \in (0, 1).
\end{equation}
Then we consider the function
$$
\tilde q(t) =
\begin{cases} 
	\xi_{\sigma_0} & \text{if } t \leq \sigma_0 - 1, \\
	(t - \sigma_0 + 1)q(\sigma_0) + ( \sigma_0 - t)\xi_{\sigma_0} & \text{if } \sigma_0 - 1 < t < \sigma_0, \\
	q(t) & \text{if } \sigma_0 \leq t.
\end{cases}
$$
Since $q \in \Gamma_c$ and $\xi_{\sigma_0} \in \mathcal{V}_c^-$, it follows from \eqref{8} that $\tilde q \in \Gamma_c$. Hence $J_c(\tilde q) \geq m_c$. Moreover, since $|q(\sigma_0) - \xi_{\sigma_0}| = r_{0}$, $ \|q\|_{L^\infty(\mathbb{R};\mathbb{R}^N)} \leq R $ we derive 
$
|\tilde q(t)|\leq  R+r_{0}$ and $\mathrm{dist}(\tilde q(t), \mathcal{V}_c^{-}) \le r_{0}$ for all $t\in(\sigma_0-1,\sigma_0)$.
Then, by the definitions of $\tilde q$, $\mu_{r_0}$ and the monotonicity of $\Phi$,
$$
J_{c,(-\infty,\sigma_0)}(\tilde q) \leq \int_{\sigma_0-1}^{\sigma_0} (\Phi(|q(\sigma_0) - \xi_{\sigma_{0}}|) + \mu_{r_{0}}) \, dt\leq \Phi(r_{0}) + \mu_{r_{0}} = \nu_0.
$$
Hence
$
J_{c,(\sigma_0,+\infty)}(\tilde q)=J_{c}(\tilde q)-J_{c,(-\infty,\sigma_0)}(\tilde q) \geq m_c - \nu_0,
$
from which it follows
$$
 J_c(q) = J_{c,(-\infty,\sigma_0)}(q) + J_{c,(\sigma_0,+\infty)}(\tilde q) \geq J_{c,(-\infty,\sigma_0)}(q) + m_c - \nu_0.
$$
Since $J_c(q)\le m_c + \nu_0$, the above inequality implies that
$J_{c,(-\infty,\sigma_0)}(q) \leq 2\nu_{0}$
which, together with \eqref{7}, yields
$$
2\nu_0 \geq J_{c,(-\infty,\sigma_0)}(q) \geq J_{c,(\sigma,\tau)}(q) \ge \lambda_{d_{\rho_{0}}} H(\tfrac{\rho_0}{16}),
$$
which contradicts \eqref{eq:nu0}.  A symmetric argument proves that
\begin{equation}\label{10}
	\text{dist}(q(t),\mathcal{V}_c^+) < \dfrac{\rho_0}{8}, \quad \forall\,\, t \geq \tau_0.
\end{equation}
Finally, to conclude the proof, we show that 
\[
-L_0 < \sigma_0 < 0 < \tau_0 < L_0.
\]
The assumption $\dist(q(0), \mathcal{V}_c) \geq \frac{\rho_0}{8}$ together with \eqref{6} and \eqref{10} implies $0 \in (\sigma_0, \tau_0)$. 
Moreover, by \eqref{5} and \eqref{eq:dr0}, $V(q(t)) - c \geq d_{r_0}$ for all $t \in (\sigma_0, \tau_0)$. Consequently,
\[m_c + \nu_0 \geq J_c(q) \geq J_{c,(0, \tau_0)}(q) 
\geq \int_{0}^{\tau_0} \bigl( V(q(t)) - c \bigr) \, dt \geq \tau_0 \, d_{r_0},
\]  
so $\tau_0 < \frac{m_c+\nu_{0}}{d_{r_0}} = L_0$. 
A similar estimate gives $\sigma_0 > -L_0$, completing the proof.
\end{proof}

Thanks to the compactness properties established above, we can now apply the direct method of the calculus of variations to obtain the existence of a minimizer for the functional $J_c$ on $\Gamma_{c}$ at level $m_c$.

\begin{lemma}\label{l7}
	There exists $q_0 \in \Gamma_c$ such that 
	$$
	J_c(q_0) = m_c,\quad \operatorname{dist} (q_0(0),\mathcal{V}_c) \geq \frac{\rho_0}{8}\quad\text{and}\quad \|q_0\|_{L^\infty (\mathbb{R}; \mathbb{R}^N)} \leq R.
	$$
\end{lemma}
\begin{proof}
By condition \hyperlink{J}{$(\mathcal{J}_c)$}, we can find a sequence $(q_n) \subset \Gamma_c$ satisfying
\begin{equation}\label{11}
	J_c(q_n)\to m_c\quad \text{and}\quad\|q_n\|_{L^\infty (\mathbb{R}; \mathbb{R}^N)} \leq R \quad \text{for all } n \in \mathbb{N}.
\end{equation}
For each $n\in\N$, since $q_{n}\in\Gamma_{c}$, by continuity there exists  $t_n\in\mathbb R$ such that
$
\dist(q_{n}(t_n), \mathcal{V}_c^{-}) = \frac{\rho_0}{8}$.
Because $\Gamma_c$ is invariant under time translations, by shifting $q_{n}$, it is not restrictive to assume $t_{n}=0$, so that
\begin{equation}\label{11.1}\dist (q_n(0),\mathcal{V}_c) \geq \frac{\rho_0}{8}\quad\forall n\in\N.\end{equation}
Without loss of generality we can also assume that
\begin{equation}\label{12}
	J_c(q_n) \leq  m_c+\nu_0\quad\forall n\in\N,
\end{equation}
with $\nu_{0}$ given by Lemma \ref{l6}.
By \eqref{11}, \eqref{11.1}, \eqref{12} and Lemma \ref{l6}, there exists a constant $L_0> 0$ such that for any $n\in\N$ 
\begin{equation}\label{13}
	\dist (q_n(t), \mathcal{V}_c^+) <\dfrac{\rho_0}{8},\quad \forall\,\, t\geq L_0\text{ and }\dist (q_n(t), \mathcal{V}_c^-) < \dfrac{\rho_0}{8} \quad \forall\,\, t\leq -L_0.
\end{equation}	
Since $\|q_{n}\|_{L^{\infty}(\R;\R^{N})}\leq R$ and $m_{c}+\nu_{0}\geq J_{c}(q_{n})\geq \int_{\R}\Phi(|\dot q_{n}(t)|)\, dt$, we obtain that $(q_{n})$ is bounded in $W^{1,\Phi}(I;\R^{N})$ for any bounded interval $I\subset\R$. Then, by the continuity of the immersion $W^{1,\Phi}_{loc}(\R;\R^{N})$ in $W^{1,1}_{loc}(\R;\R^{N})$, $(q_{n})$ is bounded in $W^{1,1}(I;\R^{N})$ and so equicontinuous on $I$ for any bounded interval $I\subset\R$.  By using a diagonal procedure, we obtain the existence of 
$q_{0}\in W^{1,\Phi}_{loc}(\mathbb R;\mathbb R^{N})$ and a subsequence of
$(q_{n})$, still denoted $(q_{n})$, such that 
$$
q_n\rightharpoonup q_0\quad \text{in}\quad W_{\text{loc}}^{1,\Phi}(\mathbb{R};\mathbb{R}^N)\quad\text{and}\quad q_n(t)\to q_0(t)\quad\text{for all }t\in\mathbb{R}.
$$ 
By Lemma \ref{l1}, the lower semicontinuity of $J_c$ yields $J_c(q_0) \leq m_c$. By \eqref{11}, \eqref{11.1}, \eqref{12}, \eqref{13} and pointwise convergence
$$
q_0\in\mathcal{X}_c,\quad \|q_0\|_{L^\infty (\mathbb{R}; \mathbb{R}^N)} \leq R,\quad \operatorname{dist} (q_0(0),\mathcal{V}_c) \geq \frac{\rho_0}{8},
$$ 
and
\begin{equation}\label{15}
	\dist (q_0(t), \mathcal{V}_c^+) \leq\dfrac{\rho_0}{8},\quad \forall\,\, t\geq L_0\text{ and }\dist (q_0(t), \mathcal{V}_c^-) \leq \dfrac{\rho_0}{8} \quad \forall\,\, t\leq -L_0.
\end{equation}
To conclude that $q_{0}\in\Gamma_{c}$ we need to verify that $\liminf_{t \to \pm\infty}\dist(q(t),\mathcal{V}_c^\pm)=0$. Consider first the case $t\to+\infty$. Since $J_{c} (q_0) \leq m_c$ we have $
\liminf_{t \to +\infty} \left(V(q_0(t)) - c\right) = 0$, so
there exists a sequence $(\tau_n)\subset(L_0,+\infty)$ with $\tau_n\to\infty$ such that $V(q_0(\tau_n))\to c$. Since $|q_0(\tau_n)|\leq R$ for any $n\in\mathbb{N}$, there exists $w\in\mathbb{R}^N$ such that, up to a subsequence, $q_0(\tau_n)\to w$. Therefore, $V(q_0(\tau_n))\to V(w)=c$. Then $w\in \mathcal{V}_c$ and, by \eqref{15}, $w\in \mathcal{V}_c^{+}$ showing that
$\liminf_{t \to +\infty} \dist (q_0(t), \mathcal{V}_c^+) = 0.$
The argument for $t\to-\infty$ is analogous, giving $\liminf_{t \to -\infty} \dist (q_0(t), \mathcal{V}_c^-) = 0.$
Therefore $q_0 \in \Gamma_c$ with $J_c(q_0)=m_c$. 
\end{proof}

%
%
\subsection{Contact times for minima in $\Gamma_{c}$ and related prescribed Energy  solutions.}

 By Lemma \ref{l7} the normalized minimal set
\[
\mathcal{M}_c = \left\{ q \in \Gamma_c \mid \dist(q(0),\mathcal{V}_c)\geq \tfrac{\rho_0}{8}, \, J_c(q) = m_c, \, \|q\|_{L^\infty(\R;\R^N)} \leq R \right\},
\]
is nonempty. According to \hyperlink{J}{$(\mathcal{J}_c)$} we have
$\mathcal{M}_c \subset \{\, q\in\Gamma_{c} \mid J_{c}(q)=m_{c}\,\}.$
Moreover, by translation $\mathcal{M}_c$ generates the set of all the minimizers with $L^\infty$-norm bounded by $R$. 
Indeed, the argument leading to \eqref{eq:1} shows that for any $q\in\Gamma_{c}$ there exists $t_q\in\R$ such that 
$
\dist(q(t_q),\mathcal{V}_c) \geq \frac{\rho_0}{8}$. 
Hence, if $J_c(q)=m_c$ and $\|q\|_{L^\infty(\R;\R^N)} \leq R$, then $\tilde q(\cdot) = q(\cdot - t_q)\in \mathcal{M}_c$.\medskip

For each $q \in \Gamma_{c}$ we define
\begin{equation}\label{Definition-alpha}
	\alpha_q = 
	\begin{cases}
		-\infty & \text{if } q(\mathbb{R}) \cap \mathcal{V}_c^- = \emptyset, \\
		\sup \{ t \in \mathbb{R} : q(t) \in \mathcal{V}_c^- \} & \text{if } q(\mathbb{R}) \cap \mathcal{V}_c^- \neq \emptyset,
	\end{cases}
\end{equation}
and
\begin{equation}\label{Definition-omega}
	\omega_q =
	\begin{cases}
		\inf \{ t > \alpha_q : q(t) \in \mathcal{V}_c^+ \} & \text{if } q(\mathbb{R}) \cap \mathcal{V}_c^+ \neq \emptyset, \\
		+\infty & \text{if } q(\mathbb{R}) \cap \mathcal{V}_c^+ = \emptyset.
	\end{cases}
\end{equation} 
Thanks to Lemma \ref{l6} $\alpha_{q}$ and $\omega_{q}$ are always well defined in $\R\cup\{\pm\infty\}$ for functions $q\in\Gamma_{c}$ with action close to $m_{c}$. Indeed
using Lemma \ref{l6} and Remark \ref{remark1}, if $q\in\Gamma_{c}$ is such that
$J_{c}(q)\leq m_{c}+\nu_{0}$,  the point $q(t)$  does not belong to $\mathcal{V}_c^+$ for $t \le -L_0$ nor to $\mathcal{V}_c^-$ for $t \ge L_0$. This guarantees that 
\begin{equation}\label{eq:bondalphaomega}
	\alpha_q \le L_0 < +\infty\quad\text{ and }\quad\omega_q \ge -L_0 > -\infty.
\end{equation}
These times are called the {\it contact times} of the trajectory $q$ with the sublevel sets $\mathcal{V}_c^+$ and $\mathcal{V}_c^-$ of the potential. 
The next lemma describes some of the first basic properties of the contact times for points in $\mathcal{M}_{c}$.

\begin{lemma}\label{l8}
	For all
	$q \in \mathcal{M}_c$, we have  
	\[\lim_{t\to\alpha_{q}^{+}}\dist(q(t),\mathcal{V}_c^-)=\lim_{t\to\omega_{q}^{-}}\dist(q(t),\mathcal{V}_c^+)=0\]  
	and, if $\alpha_q \in \mathbb{R}$ or $\omega_q \in \mathbb{R}$, then
$
q(\alpha_q) \in \mathcal{V}_c^-$ or $q(\omega_q) \in \mathcal{V}_c^+$, respectively. 
Moreover
\begin{itemize}
        \item[(a)] $V(q(t)) > c$ for all $t \in (\alpha_q, \omega_q)$;
		\item[(b)] If $\alpha_q \in \mathbb{R}$, then $q(t)= q(\alpha_q)$ for all  $t\in (-\infty, \alpha_q]$;
		\item[(c)] If $\omega_q \in \mathbb{R}$, then $q(t)= q(\omega_q)$ for all $t\in [\omega_q, +\infty)$;
		\item[(d)] $-\infty \leq \alpha_q <0< \omega_q \leq +\infty$;
\end{itemize}
and
\begin{equation}\label{20}
		m_{c}=J_c(q) = J_{c,(\alpha_q,\omega_q)}(q).
	\end{equation}
\end{lemma}
\begin{proof} 
If $\alpha_{q}=-\infty$ we have $\lim_{t\to\alpha_{q}^{+}}\dist(q(t),\mathcal{V}_c^-)=0$ by Lemma \ref{l5}. When $\alpha_{q}\in\R$, by definition there exists a sequence $(t_{n})\subset [\alpha_{q},+\infty)$ such that $t_{n}\to \alpha_{q}^{+}$ and $q(t_{n})\in \mathcal{V}_c^-$. By continuity $q(t_{n})\to q(\alpha_{q})$ and since $\mathcal{V}_c^-$ is closed, $q(\alpha_{q})\in\mathcal{V}_c^-$. This yields that if $\alpha_{q}\in\R$ we have $q(\alpha_{q})\in \mathcal{V}_c^-$ and, by continuity $\lim_{t\to\alpha_{q}^{+}}\dist(q(t),\mathcal{V}_c^-)=0$. A symmetric argument shows that $\lim_{t\to\omega_{q}^{-}}\dist(q(t),\mathcal{V}_c^+)=0$ and that $q(\omega_{q})\in \mathcal{V}_c^+$ whenever $\omega_{q}\in\R$. 

To show (a), note first that $\alpha_q < \omega_q$. Indeed, if $\alpha_q = \omega_q$, then, by \eqref{eq:bondalphaomega}, both would be finite, and by continuity we would have the contradiction 
$q(\alpha_q) = q(\omega_q) \in \mathcal{V}_c^- \cap \mathcal{V}_c^+=\emptyset$. Moreover if $t\in (\alpha_q, \omega_q)$, then, by the definition of $\alpha_q$ and $\omega_q$, $q(t)\not\in\mathcal{V}_c$ and so $V(q(t)) > c$ for all $t\in (\alpha_q, \omega_q)$ and (a) follows. 

To show (b), assume $\alpha_{q}\in\R$ and define the function 
 $\tilde q(t)$ by setting $\tilde q(t)=q(\alpha_q)$ for $t \leq \alpha_q$, and  
$\tilde q(t)=q(t)$ for $t > \alpha_q$.
Then $\tilde q(t)\in \Gamma_{c}$, so by minimality of $q$ we have
\[m_{c}\leq J_{c}(\tilde q)=J_{c, (\alpha_{q},+\infty)}(q)\leq J_{c}(q)=m_{c}.\] 
This implies that $J_{c, (-\infty,\alpha_{q})}(q)=0$, and since $V(q(t))-c\geq 0$ on $\R$, we deduce
$\int_{-\infty}^{\alpha_q} \Phi(|\dot{q}(t)|) dt = 0$.  Hence, $\dot{q}(t) = 0$ almost everywhere on $(-\infty, \alpha_q)$, and so $q$ is constantly equal to $q(\alpha_q)$ on  $(-\infty, \alpha_q)$.  A symmetric argument shows also the point (c).

From (b) and (c) we derive that $q(t)\in \mathcal{V}_c$ for $t\in\R\setminus (\alpha_{q},\omega_{q})$ and so (d) follows since $\dist(q(0),\mathcal{V}_c)\geq \frac{\rho_0}{8}$ by definition of $\mathcal{M}_c$. Finally, using (b) and (c) in the finite cases, we always have $J_c(q) = J_{c,(\alpha_q,\omega_q)}(q)$ and \eqref{20} follows.
\end{proof}

A key feature of the contact times is that the interval $(\alpha_q, \omega_q)$ identifies where the minimal trajectory $q \in \mathcal{M}_c$ is a weak solution of the system \eqref{S}.

\begin{lemma}\label{l9}
	If $q \in \mathcal{M}_c$ then 
	$q$ is a weak solution of the  system \eqref{S} on $(\alpha_q, \omega_q)$.
    \end{lemma}
\begin{proof}
Let $\psi \in C_0^\infty(\mathbb{R};\mathbb{R}^N)$ be a function with compact support contained in $(\alpha_q, \omega_q)$. Let $\operatorname{supp}\psi\subset [a, b] \subset (\alpha_q, \omega_q)$. We first show that 
\begin{equation}\label{eq:claim}\text{there exists  }h_\psi \in (0,1) \text{ such that }
\min_{t \in [a, b]} \left(V(q(t) + h\psi(t))\right) > c \quad \text{for any } h \in (0, h_\psi).
\end{equation}
To prove \eqref{eq:claim}, assume by contradiction that 
for any $\epsilon>0$ there are $\delta\in (0,\epsilon)$ and $t_\epsilon\in [a,b]$ such that 
$V(q(t_\epsilon) + \delta\psi(t_\epsilon))\leq c$. 
Choosing a sequence $\epsilon_{n} \to 0^+$ gives corresponding sequences $\delta_{n}\to 0^{+}$ and $(t_{n})\subset [a,b]$ such that $V(q(t_n) + \delta_{n}\psi(t_n))\leq c$. By compactness, along a subsequence, $t_{n}\to t_{0}\in[a,b]$, and the continuity of $V$ and $q$ yields, along this subsequence, $
 V(q(t_0))\leq c,
$
contradicting (a) of Lemma \ref{l8} and so proving \eqref{eq:claim}.

 By \eqref{eq:claim}, it follows that $q + h\psi \in \Gamma_c$ for all $h\in (0,h_\psi)$. Since $q$ minimizes $J_c$ over $\Gamma_c$, we have
\begin{equation}\label{18}
	J_c(q + h\psi) - J_c(q) \geq 0 \quad \text{for any } h \in (0, h_\psi).
\end{equation}
We prove, by Lebesgue's Theorem, that
\begin{equation}\label{19}
	\lim_{h \to 0^+}\frac{J_{c,(\alpha_q, \omega_q)}(q + h\psi) - J_{c,(\alpha_q, \omega_q)}(q)}{h}
	=\int_{\alpha_q}^{\omega_q}\left(\phi(|\dot{q}(t)|)\dot{q}(t)\dot{\psi}(t)+\nabla V(q(t))\psi(t)\right)dt.
\end{equation}
To prove \eqref{19}, we write
\[
\frac{ J_{c,(\alpha_q, \omega_q)}(q + h\psi)
      - J_{c,(\alpha_q, \omega_q)}(q)}{h}
=
\int_{\alpha_q}^{\omega_q}
\frac{\Phi(|\dot q + h\dot\psi|)-\Phi(|\dot q|)}{h}\,dt
+
\int_{\alpha_q}^{\omega_q}
\frac{V(q + h\psi)-V(q)}{h}\,dt .
\]
Since $V\in C^1(\mathbb R^N;\R)$ and $\psi$ has compact support in
$(\alpha_q,\omega_q)$, it is standard that
\[
\lim_{h\to0^+}
\int_{\alpha_q}^{\omega_q}
\frac{V(q + h\psi)-V(q)}{h}\,dt
=
\int_{\alpha_q}^{\omega_q}
\nabla V(q(t))\,\psi(t)\,dt .
\]
It therefore remains to justify that
\[
\lim_{h\to0^+}
\int_{\alpha_q}^{\omega_q}
\frac{\Phi(|\dot q + h\dot\psi|)-\Phi(|\dot q|)}{h}\,dt
=
\int_{\alpha_q}^{\omega_q}
\phi(|\dot q(t)|)\,\dot q(t)\dot\psi(t)\,dt .
\]
For a.e.\ $t\in(\alpha_q,\omega_q)$ the function
$
f_t(h)=\Phi(|\dot q(t)+h\dot\psi(t)|)
$
is of class $C^1(\mathbb R;\R)$, with derivative
\[
\frac{d}{dh}f_t(h)
=
\phi(|\dot q(t)+h\dot\psi(t)|)\,
\big(\dot q(t)+h\dot\psi(t)\big)\cdot\dot\psi(t).
\]
By the Mean Value Theorem, for every $h\in(0,h_{\psi})$ and for a.e.\ $t\in(\alpha_q,\omega_q)$ there exists
$\eta_h=\eta_h(t)$ with $0<\eta_h(t)<h$ such that
\begin{equation}\label{eq:stimadom1}
\big|
\frac{\Phi(|\dot q+h\dot\psi|)-\Phi(|\dot q|)}{h}
\big|=\big|\frac{f_t(h)-f_t(0)}{h}\big|
=\phi(|\dot q+\eta_h\dot\psi|)\,
|\dot q+\eta_h\dot\psi|\,
|\dot\psi|
\quad \text{for a.e. } t\in(\alpha_q,\omega_q).
\end{equation}
%
By \eqref{eq:phi-bound}  there exists a constant $C_0>0$ such that, for all $s\ge0$,
\begin{equation}\label{eq:phisbound}
\phi(s)s \le m\Phi(s)+C_0.
\end{equation}
Since $\eta_h<h_{\psi}< 1$, by \eqref{eq:convexity} we have furthermore
\begin{equation}\label{eq:Phi-bound}
\Phi(|\dot q+\eta_h\dot\psi|)
\le
2^{m-1}(\Phi(|\dot q|)+\Phi(|\dot\psi|)).
\end{equation}
Combining \eqref{eq:stimadom1}, \eqref{eq:phisbound} and \eqref{eq:Phi-bound}
we obtain that there exists $C>0$ such that
\[
\big|
\frac{\Phi(|\dot q+h\dot\psi|)-\Phi(|\dot q|)}{h}
\big|
\le
C(\Phi(|\dot q|)+\Phi(|\dot\psi|)+1)\,|\dot\psi|\text{ for any }h\in (0,h_{\psi}).
\]
Since $\Phi(|\dot q|)\in L^1((\alpha_q,\omega_q);\R)$ and 
$\psi\in C_0^\infty((\alpha_q,\omega_q);\R^N)$, the right hand side belongs to $L^1((\alpha_q,\omega_q);\R)$, being independent of $h$.
Therefore, by the Dominated Convergence Theorem,
\[
\lim_{h\to0^+}
\int_{\alpha_q}^{\omega_q}\frac{\Phi(|\dot q+h\dot\psi|)-\Phi(|\dot q|)}{h}\,dt =
\int_{\alpha_q}^{\omega_q}
\phi(|\dot q(t)|)\,\dot q(t)\dot\psi(t)\,dt.
\]
and \eqref{19} follows.

Now we observe that, since $\psi$ has compact support in $(\alpha_{q},\omega_{q})$, by \eqref{18}
$$
J_c(q + h\psi) - J_c(q)=J_{c,(\alpha_q, \omega_q)}(q + h\psi) - J_{c,(\alpha_q, \omega_q)}(q)\geq 0\text{ for all }h\in (0,h_{\psi}),
$$
and \eqref{19} yields
$$
\int_{\mathbb{R}}\left(\phi(|\dot{q}(t)|)\dot{q}(t)\dot{\psi}(t)+\nabla V(q(t))\psi(t)\right)dt\geq 0.
$$
Applying the same reasoning for the test function $-\psi$, we get  the reversed inequality
$$
\int_{\mathbb{R}} \left( \phi(|\dot{q}(t)|)\dot{q}(t) \dot{\psi}(t) + \nabla V(q(t)) \psi(t) \right) dt \leq 0,
$$
which allows us to conclude
$$
\int_{\mathbb{R}} \left( \phi(|\dot{q}(t)|)\dot{q}(t) \dot{\psi}(t) + \nabla V(q(t)) \psi(t) \right) dt = 0.
$$
Since $\psi \in C_0^\infty((\alpha_q,\omega_q);\mathbb{R}^N)$ is arbitrary, $q$ is a weak solution of \eqref{S} on $(\alpha_q,\omega_q)$.
\end{proof}


Starting from the solutions provided by Lemma \ref{l9}, our aim is to recover entire connecting-orbit weak solutions of \eqref{S}, possibly by employing reflection and/or periodic continuation when the contact times are finite. A key step in this construction is to show that these solutions have constant energy equal to $-c$ on the interval $(\alpha_{q},\omega_{q})$.\medskip


To be more detailed, consider the Lagrangian associate with \eqref{S}, 
$
L(q,\dot q)=\Phi(|\dot q|)+V(q)$.
The conjugate momentum is \(p=\nabla_{\dot q}L=\phi(|\dot q|)\,\dot q\) and the Hamiltonian is given by 
$H(p,q)=\widetilde{\Phi}(|p|)-V(q)$, 
where \(\widetilde{\Phi}\) is the Legendre transform, i.e., the complementary function, of \(\Phi\). Since $\Phi$ is convex, \(\widetilde{\Phi}(y)=ys-\Phi(s)\) for \(y=\Phi'(s)=\phi(s)s\). Hence, along the trajectory \(q(t)\) 
\[
\widetilde{\Phi}(|p(t)|)=\widetilde{\Phi}(\phi(|\dot q(t)|)\,|\dot q(t)|)=\phi(|\dot q(t)|)\,|\dot q(t)|^2-\Phi(|\dot q(t)|).
\] 
 This allows us to define the energy function along trajectories as 
\[
E_{q}(t)=H(p(t),q(t))=G(|\dot q(t)|)-V(q(t))\]
where 
 \[G(s)=\phi(s)\,s^2-\Phi(s).
\] 
%
%
We remark that, by \eqref{eq:Phi-growth}, $G$ satisfies 
\begin{equation}\label{ineq}
(l - 1)\Phi(s) \leq G(s)\leq (m-1)\Phi(s)\quad\text{for all}\quad s\geq0,
\end{equation}
and $G(s)=0$ if and only if $s=0$.

Note that the conservation of energy along the trajectories $q$ given by Lemma \ref{l9} cannot be deduced from the classical arguments of the calculus of variations. Indeed, our assumptions do not guarantee that the second derivative of the Lagrangian with respect to $\dot q$, $D^2_{\dot q}L(q,\dot q) = D^2_{\dot q}\Phi(|\dot q|)$, is uniformly positive definite, nor are the functions $q$ classical $C^2$ solutions of the system. Our variational framework, however, allows us to establish both a weak form of the conservation of energy and its exact value by considering suitable variations of the action along dilations of $q$.
 
\begin{lemma}\label{l12}
	For every $q \in \mathcal{M}_c$, it turns out that $E_q(t)=-c$ for a.e. $t\in(\alpha_q, \omega_q)$.   
\end{lemma}
\begin{proof}
Let $q \in \mathcal{M}_c$. We begin by proving that for any $r,h \in (\alpha_q, \omega_q)$ with $r < h$, it holds 
\begin{equation}\label{NEQW}
	\int_{r}^{h}(V(q(t)) - c)\, dt = \int_{r}^{h}G(|\dot{q}(t)|)\, dt.
\end{equation}
We consider the translated function $\tilde q=q(\cdot + r)$ observing that $\tilde q\in \Gamma_{c}$ and $J_{c}(\tilde q)=m_{c}$.
For any $s > 0$ we define the function
$$
q_s(t) = 
\begin{cases} 
	\tilde q(t), & t \leq 0, \\
	\tilde q\left(\frac{t}{s}\right), & t > 0.
\end{cases}
$$
By construction, 
$
q_s \in \Gamma_c$, $\alpha_{q_s} = \alpha_{\tilde q}$ and $\omega_{q_s} = s \omega_{\tilde q}$. 
Then by \eqref{20}
$$
J_{c,(\alpha_{\tilde q}, s\omega_{\tilde q})}(q_s) \geq m_c = J_{c,(\alpha_{\tilde q}, \omega_{\tilde q})}(\tilde q).
$$
This inequality implies
$$
0 \leq \int_0^{s\omega_{\tilde q}}( \Phi(|\dot{q_s}(t)|) + (V(q_s(t)) - c))\, dt - \int_0^{\omega_{\tilde q}}( \Phi(|\dot{{\tilde q}}(t)|) + (V({\tilde q}(t)) - c))\, dt.
$$
Since $\dot{q_s}(t) = \frac{1}{s} \dot{{\tilde q}}(\frac{t}{s})$ for $t>0$, making the change of variable $x = \frac{t}{s}$, we have
$$
0 \leq \int_0^{\omega_{\tilde q}} \left( s\Phi(\tfrac{1}{s}|\dot{{\tilde q}}(t)|) - \Phi(|\dot{{\tilde q}}(t)|)  + (s - 1)(V({\tilde q}(t)) - c)\right)\, dt.
$$
Denoting by $f(s)$ the right hand side of the above inequality 
we have $f(s) \geq 0$ for all $s > 0$ and $f(1) = 0$ so that $1$ is a minimum of $f$. We claim that $f$ is differentiable at $s=1$ with 
\begin{equation}\label{eq:claimfprimo}
	f'(1)=\int_0^{\omega_{\tilde q}}\left(\Phi(|\dot {\tilde q}|)-\phi(|\dot {\tilde q}|)|\dot {\tilde q}|^2+(V({\tilde q})-c)\right)dt.
\end{equation}
For $s\neq 1$ we  write the incremental quotient
\[
\frac{f(s)}{s-1}
=\int_0^{\omega_{\tilde q}}\big(
\frac{s\,\Phi(\tfrac{1}{s}|\dot {\tilde q}|)-\Phi(|\dot {\tilde q}|)}{s-1}
+(V({\tilde q})-c)\big)dt.
\]
We prove that there exists $C>0$ such that for every $a\ge 0$ and every
$s\in[1/2,2]$,
\begin{equation}\label{eq:domination}
\big|\frac{s\,\Phi(a/s)-\Phi(a)}{s-1}\big|\le C\,\Phi(a).
\end{equation}
If $a=0$ the inequality is trivial. Let $a>0$ and set
$F(s)=s\,\Phi(a/s)$ for $s>0$. Then $F\in C^1((0,+\infty);\R)$ and
$F'(s)=\Phi(a/s)-\tfrac{1}{s^2}\phi(a/s)a^2$.
By the mean value theorem, for every $s\neq 1$ there exists
$\theta=\theta(s)$ such that
\[
\frac{F(s)-F(1)}{s-1}
=F'(\theta)=\Phi(a/\theta)-\frac{1}{\theta^2}\phi(a/\theta)a^2.
\]
If $s\in[1/2,2]$, then also $\theta\in[1/2,2]$.  
By \eqref{eq:Phi-growth} we have
$\phi(a/\theta)(a/\theta)^2\le m\,\Phi(a/\theta)$, and so
since
$\theta\in[1/2,2]$,  
\eqref{eq:omogeneity} gives
\[
\big|\Phi(a/\theta)-\frac{1}{\theta^2}\phi(a/\theta)a^2\big|
\le (1+m)\Phi(a/\theta)
\le (1+m)2^{m}\,\Phi(a),
\]
which proves \eqref{eq:domination}.

For $a=|\dot {\tilde q}(t)|$, 
\eqref{eq:domination} yields that for any $s\in[1/2,2]$
\[\big|\frac{s\,\Phi(|\dot {\tilde q}(t)|/s)-\Phi(|\dot {\tilde q}(t)|)}{s-1}\big|\le C\,\Phi(|\dot {\tilde q}|),
\text{ for a.e. }t\in(0,\omega_{\tilde q}).\]
Since $J_c(q)=m_c<\infty$, we have $\Phi(|\dot {\tilde q}|)\in L^1((0,\omega_{\tilde q});\R)$ and also
$V({\tilde q})-c\in L^1((0,\omega_{\tilde q});\R)$. Then, by dominated convergence 
\[
\lim_{s\to 1}\frac{f(s)}{s-1}= \int_0^{\omega_{\tilde q}}\Big(\Phi(|\dot {\tilde q}|)-\phi(|\dot {\tilde q}|)|\dot {\tilde q}|^2+(V({\tilde q})-c)\Big)dt
\]
proving \eqref{eq:claimfprimo}. Since $1$ is a minimum point of $f$ we have $f'(1) = 0$ and so 
\[
\int_0^{\omega_{\tilde q}}\big(V({\tilde q}(t))-c\big)\,dt=\int_0^{\omega_{\tilde q}}\big(\phi(|\dot {\tilde q}(t)|)|\dot {\tilde q}(t)|^2-\Phi(|\dot {\tilde q}(t)|)\big)\,dt
=\int_0^{\omega_{\tilde q}}G(|\dot {\tilde q}(t)|)\,dt,
\]
where the last equality follows from the definition of $G$. 

Returning to the original function $q={\tilde q}(\cdot-r)$, this shows that $
\int_r^{\omega_q} (V(q(t)) - c)\, dt = \int_r^{\omega_q} G(|\dot{q}(t)|)\, dt,
$ for any $r\in (\alpha_q, \omega_q)$, and since $ \int_{r}^{h}=\int_{r}^{\omega_q}-\int_{h}^{\omega_q}$,
\eqref{NEQW}  follows. 

Dividing both sides of the equality in \eqref{NEQW} by $h - r$ gives
$$
\frac{1}{h -r} \int_{r}^{h} (V(q(t)) - c)\, dt = \frac{1}{h - r} \int_{r}^{h} G(|\dot{q}(t)|)\, dt.
$$
Therefore, since both $V(q(t)) - c$ and $G(|\dot{q}(t)|)$ belong to $L^{1}((\alpha_q, \omega_q);\R)$, the Lebesgue point Theorem 
yields
$$
V(q(t_0)) - c = G(|\dot{q}(t_0)|) \text{ for a.e. }t_{0}\in (\alpha_q, \omega_q) .
$$
Consequently, $E_q(t) = -c$ a.e. on $(\alpha_q, \omega_q)$, completing the proof of the lemma.
\end{proof}

 By item (a) of Lemma \ref{l8}, if $q \in \mathcal{M}_c$ then $V(q(t)) > c$ for all $t \in (\alpha_q, \omega_q)$. Hence, Lemma \ref{l12} yields
\[
G(|\dot q(t)|) > 0 \quad \text{a.e. on } (\alpha_q, \omega_q).
\]
Since $G(s)=0$ if and only if $s=0$, we conclude that
\begin{equation*}\label{eq:dotqnot=0}
\dot q(t)\neq 0 \quad \text{for a.e. } t \in (\alpha_q, \omega_q).
\end{equation*}
Thanks to these properties, and since by Lemma \ref{l9} $q$ is a weak solution of \eqref{S} on $(\alpha_q, \omega_q)$, we are in a position to improve the regularity of $q$.  
More precisely, the next lemma shows that, thanks to the local
invertibility of the map $g(x)=\phi(|x|)x$ (see Lemma~\ref{lem:ginvC1}), $q$ is a classical solution of \eqref{S} on $(\alpha_q,\omega_q)$, with constant energy
$E_{q}(t)=-c$ for all $t \in (\alpha_q,\omega_q)$.
\begin{lemma}\label{l233}
Assume \hyperlink{phi1}{$(\phi_1)$} and \hyperlink{phi2}{$(\phi_2)$}. 
If $q \in \mathcal{M}_c$, then 
$q \in C^2\big((\alpha_q,\omega_q); \mathbb{R}^N\big)$ and solves \eqref{S} on $(\alpha_q, \omega_q)$,
with $E_{q}(t)=-c$ for all $t\in (\alpha_q, \omega_q)$.
\end{lemma}

\begin{proof}
Let $q \in \mathcal{M}_c$. 
Recall that by Lemma~\ref{lem:ginvC1} the function $g:\mathbb{R}^N\to\mathbb{R}^N$,
defined by $g(x)=\phi(|x|)x$ for $x\neq0$ and $g(0)=0$,
belongs to $C(\mathbb{R}^N,\mathbb{R}^N)\cap C^1(\mathbb{R}^N\setminus\{0\};\mathbb{R}^N)$. Since
$q$ is a weak solution of \eqref{S}, we have
\[
(g(\dot q(t)))' = \nabla V(q(t)) \quad \text{a.e. on } (\alpha_q,\omega_q).
\]
In particular $g(\dot q(\cdot))$ is absolutely continuous on $(\alpha_q,\omega_q)$ and so, for any $t_0, t \in (\alpha_q,\omega_q)$,
\begin{equation}\label{eq:representation}
g(\dot q(t)) = g(\dot q(t_0)) + \int_{t_0}^t \nabla V(q(s))\, ds.
\end{equation} 

Since by Lemma \ref{l8} we have $V(q(t)) > c$ for all $t \in (\alpha_q, \omega_q)$, by continuity we have that, fixed any interval $[a,b]\subset (\alpha_q,\omega_q)$, there exists $\delta=\delta(a,b)>0$ such that $V(q(t)) > c+\delta$ for all $t \in [a, b]$. Then, by Lemma \ref{l12} we obtain
$
G(|\dot q(t)|) \geq \delta$ a.e. on $[a, b]$. Since $G$ is continuous and $G(s)=0$ if and only if $s=0$ we deduce that there exists $\eta=\eta(\delta)>0$ such that $G(s)\geq\delta$ implies $s\geq \eta$, so that 
\begin{equation}\label{eq:dotqdifferentzero}|\dot q(t)|\geq \eta\text{ a.e. on }[a,b].\end{equation}
Since $g$ is continuous on $\mathbb{R}^N$ and $g(x)=0$ if and only if $x=0$ (by \hyperlink{phi1}{$(\phi_1)$}), from \eqref{eq:dotqdifferentzero} we deduce that there exists $\epsilon=\epsilon(\eta)>0$ such that
\[|g(\dot q(t))|\geq\epsilon\text{ a.e. on }[a,b]\]
and since the function $g(\dot q(t))$ is continuous on $(\alpha_q,\omega_q)$,
\begin{equation}\label{eq:g}|g(\dot q(t))|\geq\epsilon\text{ for all }t\in [a,b].\end{equation}
To conclude the proof, we exploit Lemma~\ref{lem:ginvC1}, which guarantees
that $g$ is a local diffeomorphism on
$\mathbb{R}^N\setminus\{0\}$, and that its inverse $g^{-1}$ is of class $C^1$
on any compact subset of $\mathbb{R}^N\setminus\{0\}$.

Fix any interval $[a,b]\subset(\alpha_q,\omega_q)$.
From \eqref{eq:g} we have
\[
|g(\dot q(t))|\ge \varepsilon \qquad \text{for all } t\in[a,b],
\]
so that, by continuity, $g(\dot q([a,b]))$ is a compact subset of $\mathbb{R}^N\setminus B_\varepsilon(0)$.
Recalling that $g(\dot q(\cdot))$ satisfies
\eqref{eq:representation}, we can define
\[
v(t):=g^{-1}\!\left(g(\dot q(t_0))+\int_{t_0}^t \nabla V(q(s))\,ds\right),
\qquad t\in[a,b],
\]
for some fixed $t_0\in[a,b]$.
By Lemma~\ref{lem:ginvC1}, the inverse map $g^{-1}$ is of class $C^1$
on compact subsets of $\mathbb{R}^N\setminus\{0\}$.
Since, as observed above, $g(\dot q([a,b]))$ is a compact subset of
$\mathbb{R}^N\setminus\{0\}$, it follows that
$v \in C^1([a,b];\mathbb{R}^N)$, and
\[
\dot q(t)=v(t)\qquad \text{for a.e. } t\in[a,b].
\]
Since $q$ is absolutely continuous, it follows that
\[
q(t)=q(t_0)+\int_{t_0}^t v(s)\,ds \qquad \text{for all } t\in[a,b],
\]
and hence $q\in C^2([a,b];\mathbb{R}^N)$.
As $[a,b]\subset(\alpha_q,\omega_q)$ was arbitrary, we conclude that
\[
q\in C^2\big((\alpha_q,\omega_q);\mathbb{R}^N\big),
\]
and therefore $q$ is a classical solution of \eqref{S} on $(\alpha_q,\omega_q)$. By the regularity of $q$, we moreover deduce that the function $E_q(t)$ is continuous on
$(\alpha_q,\omega_q)$. Since $E_q(t)=-c$ a.e. on $(\alpha_q,\omega_q)$, it follows that
$E_q(t)=-c$ for all $t\in(\alpha_q,\omega_q)$, concluding the proof of the lemma.
\end{proof}

We are now in a position to state and prove the main result of this section.
Let us recall the definition of bounded connecting orbit.

\begin{definition}\label{def:minconnection}
Assume that condition \hyperlink{V}{$(\mathcal{V}_c)$}holds for some $c\in\R$.
Let $I=(\alpha,\omega)\subset\R$, with $-\infty\le \alpha<\omega\le +\infty$, and denote by $\bar I$ the closure of $I$.
We say that $q:\bar I\to\R^N$ is a {\sl $c$--connection between  $\mathcal V_c^-$ and $\mathcal V_c^+$ with contact times $\alpha$ and $\omega$} if $q\in L^\infty(I;\R^N)$ and 
\begin{itemize}
\item[(i)] $q\in C^{1}(\bar I;\R^N)\cap C^{2}(I;\R^N)$ and solves \eqref{S} on $I$;

\item[(ii)] $V(q(t))>c$ and $E_q(t)=G(|\dot q(t)|)-V(q(t))=-c$ for all $t\in I$;

\item[(iii)]
$
\lim_{t\to\alpha^+}\dist(q(t),\mathcal V_c^-)=
\lim_{t\to\omega^-}\dist(q(t),\mathcal V_c^+)=0;
$\end{itemize}
We say moreover that $q$ is a {\sl minimal} $c$--connection between  $\mathcal V_c^-$ and $\mathcal V_c^+$ if it also satisfies
\begin{itemize}
\item[(iv)] $J_{c,I}(q)=m_c $.
\end{itemize}

\end{definition}
\begin{remark}\label{R:connectors} Note that if $I=(\alpha,\omega)\subset\R$ and $q:\bar I\to\mathbb R^{N}$ is a $c$--connection between  $\mathcal V_c^-$ and $\mathcal V_c^+$ then it  satisfies $\lim_{t\to\alpha^+}|\dot q(t)|=
\lim_{t\to\omega^-}|\dot q(t)|=0.$

Indeed, by (iii) we have $\lim_{t\to\alpha^+}\dist(q(t),\mathcal V_c^-)=0$.
Hence, for any sequence $t_n \to \alpha^{+}$, we have $\dist(q(t_n), \mathcal{V}_c^-) \to 0$. 
Since $q\in L^\infty(I;\R^N)$ and $\mathcal{V}_c^-$ is closed, up to a subsequence $q(t_n) \to \xi \in \mathcal{V}_c^-$. 
Then $V(q(t_n)) \to V(\xi) \le c$ while $V(q(t_n)) \ge c$, so $V(q(t_n)) \to c$.
By (ii), $G(|\dot q(t_n)|)=V(q(t_n))-c\to 0$, and by \eqref{ineq} this implies $\dot q(t_n)\to 0$. 
Since $(t_n)$ was arbitrary, this yields $\dot q(t)\to 0$ as $t\to\alpha^+$.

An analogous reasoning applies to show also that $\dot q(t)\to 0$ as $t\to\omega^-$.
\end{remark}

The analysis carried out above naturally leads to the following existence theorem.
\begin{theorem}\label{thm:minimal-connection}
Assume \hyperlink{phi1}{$(\phi_1)$}--\hyperlink{phi2}{$(\phi_2)$}. Let $V\in C^1(\R^N;\R)$ and let $c\in\R$ be such that
the conditions \hyperlink{V}{$(\mathcal{V}_c)$} and \hyperlink{J}{$(\mathcal{J}_c)$} hold.
Then \eqref{S} has a minimal c-connection between $\mathcal V_c^-$ and $\mathcal V_c^+$.
\end{theorem}
\begin{proof}
By Lemma~\ref{l7} there exists $q_0\in\mathcal M_c$ with $\|q_{0}\|_{L^{\infty}(\mathbb R;\mathbb R^{N})}\leq R$. Let $\alpha=\alpha_{q_{0}}$ and $\omega=\omega_{q_{0}}$ denote the contact times defined in \eqref{Definition-alpha}--\eqref{Definition-omega}, and set $I=(\alpha,\omega)$. Consider the restriction $q:=q_0|_{\bar I}$. By Lemma~\ref{l233}, $q\in C^2(I;\R^N)$ and solves \eqref{S} on $I$, with $E_q(t)=-c$ for all $t\in I$. By Lemma~\ref{l8} we moreover have $J_{c,I}(q)=m_c$, $V(q(t))>c$ for all $t\in I$, and $\lim_{t\to\alpha^+}\dist(q(t),\mathcal{V}_c^-)=\lim_{t\to\omega^-}\dist(q(t),\mathcal{V}_c^+)=0$. 

We claim that $q$ is a minimal $c$--connection between 
$\mathcal V_c^-$ and $\mathcal V_c^+$. 
To this end, by Definition~\ref{def:minconnection},
Remark~\ref{R:connectors} and the properties already established above, 
it remains to prove that $q\in C^{1}(\bar I;\R^N)$. 

Since $q\in C^2(I;\R^N)$ and $q$ is continuous on $\bar I$, 
it suffices to verify the existence of the one-sided derivatives 
at the the finite endpoints of $I$. 
More precisely, if $\alpha\in\R$ (resp. $\omega\in\R$), 
we have to show that $D_{+}q(\alpha)=0$ 
(resp. $D_{-}q(\omega)=0$).  

If $\alpha\in\R$, the continuity of $q$ at $\alpha$ together with 
$\lim_{t\to\alpha^+}\dot q(t)=0$ implies that the right derivative 
$D_{+}q(\alpha)$ exists and equals $0$. 
The argument in the case $\omega\in\R$ is analogous.
\end{proof}

\subsection{From $c$--connections to entire connecting orbits}\label{Subsec3.3} 
In the previous subsections we proved the existence of minimal $c$--connections
between $\mathcal V_c^-$ and $\mathcal V_c^+$ (Theorem~\ref{thm:minimal-connection})
and established their main structural properties.
In particular, any minimal $c$--connection is a classical solution of \eqref{S}
on its interior interval and has constant energy $-c$.

In this subsection we  prove that every $c$--connection can be extended to an entire
weak solution of \eqref{S} on $\R$.
Depending on whether none, one, or both contact times are finite,
the resulting global solution is respectively of heteroclinic, homoclinic,
or brake type.

\begin{remark}\label{R:nonunique-extension}
In the quasilinear case, the Cauchy problem for \eqref{S} may fail to
be uniquely solvable at points where $\dot q(t_0)=0$ (in particular when the
map $v\mapsto \phi(|v|)v$ is degenerate or singular at $v=0$).
Therefore, global continuations across contact times need not be unique in general.
In what follows we adopt the canonical extension obtained by time--reflection
at each finite contact time (and periodic continuation when both contact times
are finite), which is natural in view of the symmetry induced by the energy
constraint.
\end{remark}

Let $q:\bar I\to\R^N$ be a $c$--connection, where $I=(\alpha,\omega)$
with $-\infty\le \alpha<\omega\le+\infty$.
We shall continue to refer to $\alpha$ and $\omega$ as the contact times of $q$.\medskip 

We first consider the case in which the contact times are infinite,
that is $\alpha=-\infty$ and $\omega=+\infty$.
In this situation the connection is already defined on the whole line and,
by the structural properties of $c$--connections established above,
\[
q\in C^{2}(\R;\R^{N})
\]
is a classical entire solution of \eqref{S} on $\R$ satisfying, by definition and Remark~\ref{R:connectors},
\[
E_q(t)=-c \quad\text{for all }t\in\R, \qquad
V(q(t))>c \quad\text{for all }t\in\R,
\]
\[
\lim_{t\to -\infty}\dist(q(t),\mathcal V_c^-)=0,
\qquad
\lim_{t\to +\infty}\dist(q(t),\mathcal V_c^+)=0,
\qquad
\lim_{t\to\pm\infty}\dot q(t)=0.
\]

We refer to such solutions as {\sl heteroclinic--type},
since they have no finite contact times and connect the two sets
$\mathcal V_c^-$ and $\mathcal V_c^+$ at infinity.


\medskip
We now consider the case in which at least one contact time is finite.
In this situation an entire connecting solution can be obtained by
reflecting the $c$--connection at the finite contact time(s) and,
if necessary, by periodic continuation on $\R$.

More precisely, let $q$ be a $c$--connection with contact times
$\alpha$ and $\omega$.

If $\alpha>-\infty$, we define the reflected map
\begin{equation}\label{NewFunction-}
q_-(t)=
\begin{cases}
q(2\alpha-t) & \text{if } t\in(2\alpha-\omega,\alpha),\\[4pt]
q(t) & \text{if } t\in[\alpha,\omega),
\end{cases}
\end{equation}
which extends $q$ to the left of $\alpha$.

If $\omega<+\infty$, we define analogously
\begin{equation}\label{NewFunction+}
q_+(t)=
\begin{cases}
q(t) & \text{if } t\in(\alpha,\omega],\\[4pt]
q(2\omega-t) & \text{if } t\in(\omega,2\omega-\alpha),
\end{cases}
\end{equation}
which extends $q$ to the right of $\omega$.

The functions $q_\pm$ inherit from $q$ the following properties.

\begin{lemma}\label{l13}
Let $q$ be a $c$--connection with contact times $\alpha$ and $\omega$.

\begin{enumerate}
\item[(a)] If $\alpha>-\infty$, then
\begin{itemize}
\item[(i)] $q_-\in C^{2}((2\alpha-\omega,\omega)\setminus\{\alpha\};\R^N)$ and it solves \eqref{S} on 
$(2\alpha-\omega,\alpha)\cup(\alpha,\omega)$;

\item[(ii)] $q_-\in C^{1}((2\alpha-\omega,\omega);\R^N)$, 
$\dot q_-(\alpha)=0$, $q_-(\alpha)\in\mathcal V_c^{-}$, 
$V(q_-(t))>c$ on $(2\alpha-\omega,\omega)\setminus\{\alpha\}$, and moreover
\[
\dot q_-(t)\to 0 \ \text{and}\ 
\dist(q_-(t),\mathcal V_c^{+})\to 0
\quad \text{as } t\to (2\alpha-\omega)^+ \text{ or } t\to \omega^-;
\]

\item[(iii)] $q_-$ is a weak solution of \eqref{S} with 
$E_{q_-}(t)=-c$ on $(2\alpha-\omega,\omega)$.
\end{itemize}

\item[(b)] If $\omega<+\infty$, then
\begin{itemize}
\item[(i)] $q_+\in C^{2}((\alpha,2\omega-\alpha)\setminus\{\omega\};\R^N)$ and it solves \eqref{S} on 
$(\alpha,\omega)\cup(\omega,2\omega-\alpha)$;

\item[(ii)] $q_+\in C^{1}((\alpha,2\omega-\alpha);\R^N)$,
$\dot q_+(\omega)=0$, $q_+(\omega)\in\mathcal V_c^{+}$,
$V(q_+(t))>c$ on $(\alpha,2\omega-\alpha)\setminus\{\omega\}$, and moreover
\[
\dot q_+(t)\to 0 \ \text{and}\ 
\dist(q_+(t),\mathcal V_c^{-})\to 0
\quad \text{as } t\to \alpha^+ \text{ or } t\to (2\omega-\alpha)^-;
\]

\item[(iii)] $q_+$ is a weak solution of \eqref{S} with 
$E_{q_+}(t)=-c$ on $(\alpha,2\omega-\alpha)$.
\end{itemize}
\end{enumerate}
\end{lemma}

\begin{proof}
We prove (a), the argument for (b) is symmetric. 

Suppose that $\alpha > -\infty$. By \eqref{NewFunction-}, $q_-$
coincides with $q$ on $(\alpha,\omega)$ and with its reflection on
$(2\alpha-\omega,\alpha)$, both $C^{2}$ solutions of \eqref{S}.
This proves (i).


Point (ii) follows by symmetry and by the properties of $q$. 
Indeed, since $q$ is continuous at $\alpha$, the same holds for $q_-$. 
Moreover, by Remark~\ref{R:connectors} we have 
$\dot q(t)\to 0$ as $t\to\alpha^{+}$ and, by symmetry,
$\dot q_-(t)\to 0$ as $t\to\alpha^{-}$. 
Hence $\dot q_-(\alpha)=0$ and $q_-\in C^1((2\alpha-\omega,\omega);\R^N)$.

Since $\dist(q(t),\mathcal V_c^-)\to 0$ as $t\to\alpha^{+}$ and
$q$ is continuous at $\alpha$, we obtain $q(\alpha)\in\mathcal V_c^-$. 
Hence $q_-(\alpha)=q(\alpha)\in\mathcal V_c^-$.

By definition of $c$--connection we have $V(q(t))>c$ for all $t\in(\alpha,\omega)$,
and since $q_-$ is obtained by time reflection, we also obtain
$V(q_-(t))>c$ for all $t\in(2\alpha-\omega,\alpha)$.

Finally, by Remark~\ref{R:connectors}
$\dot q(t)\to 0$ and $\dist(q(t),\mathcal V_c^+)\to 0$ as $t\to\omega^{-}$.
Hence, by symmetry of the construction, we obtain 
$\dot q_-(t)\to 0$ and $\dist(q_-(t),\mathcal V_c^+)\to 0$
also as $t\to(2\alpha-\omega)^{+}$ and (ii) is proved.


To prove (iii) we observe that, by (i), $q_-$ is classical solution of \eqref{S} on each of the open interval $(2\alpha_q-\omega_q,\alpha_q)$ and $(\alpha_q,\omega_q)$.  
Then $(\phi(|\dot q_{-}(t)|)\dot q_{-}(t))'=\nabla V(q_{-}(t))$ on $(2\alpha_q-\omega_q,\omega_q)\setminus \{\alpha_{q}\}$. Hence for any test function $\psi\in C_{0}^{\infty}((2\alpha_q - \omega_q,\omega_{q});\R^{N})$ and for $\epsilon>0$ small, integrating by parts on $(2\alpha-\omega,\alpha-\varepsilon)$ we obtain
\begin{align*}0&=\int_{2\alpha_q-\omega_q}^{\alpha-\epsilon}-(\phi(|\dot q_{-}(t)|)\dot q_{-}(t))'\psi(t)+\nabla V(q_{-}(t))\psi(t)\, dt\\
&=\int_{2\alpha-\omega}^{\alpha-\epsilon}\phi(|\dot q_{-}(t)|)\dot q_{-}(t)\dot \psi(t)+\nabla V(q_{-}(t))\psi(t)\, dt+
\phi(|\dot{q}_{-}(\alpha-\epsilon)|)\dot{q}_{-}(\alpha-\epsilon)\psi(\alpha
-\epsilon)
.\end{align*}
Letting $\epsilon\to 0^{+}$, since $\dot q_{-}(t)\to 0$ as $t\to\alpha$, we conclude
$$0=\int_{2\alpha-\omega}^{\alpha}\phi(|\dot q_{-}(t)|)\dot q_{-}(t)\dot \psi(t)+\nabla V(q_{-}(t))\psi(t)\, dt.$$
The same argument on $(\alpha+\varepsilon,\omega)$ yields
$$0=\int_{\alpha}^{\omega}\phi(|\dot q_{-}(t)|)\dot q_{-}(t)\dot \psi(t)+\nabla V(q_{-}(t))\psi(t)\, dt.$$
Summing the two identities, we obtain that for every
$\psi\in C_{0}^{\infty}((2\alpha-\omega,\omega);\R^{N})$
$$0=\int_{2\alpha - \omega_q}^{\omega}\phi(|\dot q_{-}(t)|)\dot q_{-}(t)\dot \psi(t)+\nabla V(q_{-}(t))\psi(t)\, dt.$$
Hence $q_{-}$ is a weak solution of \eqref{S} on $(2\alpha-\omega,\omega)\).

Finally, the energy $E_{q_-}(t)$ is constant and equal to $-c$
on $(\alpha,\omega)$ by the defining properties of a $c$--connection.
By symmetry, the same holds on $(2\alpha-\omega,\alpha)$.
Hence, since $q_-\in C^1((2\alpha-\omega,\omega);\R^N)$,
we conclude that
$
E_{q_-}(t)=-c$ for all $t\in(2\alpha-\omega,\omega)$.
\end{proof}

In particular, if a $c$--connection has exactly one finite contact time,
the corresponding reflected trajectory provides an entire weak solution
of \eqref{S} on $\R$.

More precisely, as a direct consequence of Lemma~\ref{l13}, we obtain the
following result.
%

\begin{proposition}\label{p:homoclinic}
Let $q$ be a $c$--connection with contact times $\alpha$ and $\omega$.

\begin{itemize}

\item[(a)] If $\alpha\in\R$ and $\omega=+\infty$, then  
$q_c \equiv q_- \in C^1(\R;\R^N)\cap C^2(\R\setminus\{\alpha\};\R^N)$  
is an entire weak solution of \eqref{S}, classical on 
$(-\infty,\alpha)$ and $(\alpha,+\infty)$. 
Moreover, $q_c$ has constant energy $E_{q_c}\equiv -c$ and satisfies  
\[
\dot q_c(\alpha)=0,\quad 
q_c(\alpha)\in\mathcal V_c^-,\quad 
V(q_c(t))>c \text{ if } t\neq\alpha,\quad 
q_c(\alpha-t)=q_c(\alpha+t)\text{ on }\R,
\]
together with the asymptotic properties
\[
\lim_{t\to\pm\infty}\dot q_c(t)=0,
\qquad 
\lim_{t\to\pm\infty}\dist(q_c(t),\mathcal V_c^+)=0.
\]

\item[(b)] If $\alpha=-\infty$ and $\omega\in\R$, then  
$q_c \equiv q_+ \in C^1(\R;\R^N)\cap C^2(\R\setminus\{\omega\};\R^N)$  
is an entire weak solution of \eqref{S}, classical on 
$(-\infty,\omega)$ and $(\omega,+\infty)$. 
Moreover, $q_c$ has constant energy $E_{q_c}\equiv -c$ and satisfies
\[
\dot q_c(\omega)=0,\quad 
q_c(\omega)\in\mathcal V_c^+,\quad 
V(q_c(t))>c \text{ if } t\neq\omega,\quad 
q_c(\omega-t)=q_c(\omega+t)\text{ on }\R,
\]
together with the asymptotic properties
\[
\lim_{t\to\pm\infty}\dot q_c(t)=0,
\qquad 
\lim_{t\to\pm\infty}\dist(q_c(t),\mathcal V_c^-)=0.
\]

\end{itemize}
\end{proposition}

We refer to these solutions as {\sl homoclinic--type} solutions. 
They are obtained by reflection when only one contact time is finite. 
They have constant energy $-c$, are symmetric with respect to the
finite contact time at which they touch one of the sets
$\mathcal V_c^{\pm}$, and are asymptotic to the other set with
vanishing velocity at infinity.

\medskip

We now consider the case in which both contact times are finite. 
In this situation, the reflected solution can be further continued by periodic
extension to an entire periodic weak solution of \eqref{S}. 
More precisely

\begin{proposition}\label{p:brake} 
Let $q$ be a $c$--connection with contact times $\alpha,\omega\in\R$
and $\alpha<\omega$.
Then there exists a $T=2(\omega-\alpha)$--periodic weak solution 
$q_c\in C^1(\R;\R^N)$ to \eqref{S} satisfying the energy identity
\[
E_{q_c}(t)=G(|\dot q_c(t)|)-V(q_c(t))=-c 
\quad \text{for all } t\in\R,
\]
and the following properties:
\begin{enumerate}
\item[(a)] $q_c=q$ on $[\alpha,\omega]$;

\item[(b)] $V(q_c(\alpha))=V(q_c(\omega))=c$ and 
$V(q_c(t))>c$ for every $t\in(\alpha,\omega)$;

\item[(c)] $q_c(\alpha)\in\mathcal V_c^-$,
$q_c(\omega)\in\mathcal V_c^+$, and
$\dot q_c(\alpha)=\dot q_c(\omega)=0$;

\item[(d)] $q_c$ is symmetric with respect to $\alpha$ and $\omega$, namely
\[
q_c(\alpha-t)=q_c(\alpha+t),\qquad
q_c(\omega-t)=q_c(\omega+t)\quad \forall t\in\R .
\]
\end{enumerate}
\end{proposition}
\begin{proof}
Let $q$ be a $c$--connection with $-\infty<\alpha<\omega<+\infty$. 
By Lemma~\ref{l13}, $q_+$ is a weak solution of \eqref{S} on $(\alpha,2\omega-\alpha)$, classical on $(\alpha,\omega)$ and $(\omega,2\omega-\alpha)$. Since $q$ is continuous on $[\alpha,\omega]$, we can extend $q_{+}$ continuously at $\alpha$ and $2\omega-\alpha$ by setting 
$q_{+}(\alpha)=q_{+}(2\omega-\alpha)=q(\alpha)$.
By symmetry and Remark~\ref{R:connectors} we obtain
\begin{equation}\label{eq:d+-}
0=D_{+}q(\alpha)=D_{+}q_+(\alpha)=D_{-}q_+(2\omega-\alpha).
\end{equation}
Since $q_{+}(\alpha)=q_{+}(2\omega-\alpha)$, the periodic extension with period 
$T=2(\omega-\alpha)$,
\[
q_c(t+kT)=q_+(t), \qquad \forall t\in[\alpha,\alpha+T),\ k\in\mathbb Z,
\]
is continuous on $\R$. Moreover, by \eqref{eq:d+-} we have 
$q_c\in C^{1}(\R;\R^{N})$ with 
$\dot q_c(\alpha)=\dot q_c(\omega)=0$.

By construction and Lemma~\ref{l13}, $q_c$ satisfies properties (a)--(d) and the energy identity. It remains to check that $q_c$ is a weak solution of \eqref{S} on $\R$. To this aim we argue as in Lemma~\ref{l13}. 

Let $\psi\in C_0^\infty(\R;\R^N)$.
For any $k\in\mathbb Z$, $q_c$ coincides with a translation of $q_+$ on the interval $[\alpha+kT,2\omega-\alpha+kT]$.
Hence $q_c$ 
is a classical solution on the subintervals $A_{k}=(\alpha+kT,\omega+kT)$ and $B_{k}=(\omega+kT,2\omega-\alpha+kT)$.

Multiplying the equation $(\phi(|\dot q_c|)\dot q_c)'=\nabla V(q_c)$ by $\psi$ and integrating by parts on $A_{k}$ and $B_{k}$, the boundary terms at the endpoints vanish because $\dot q_c=0$ there. Thus, for any $k\in\mathbb Z$
\[
\int_{A_{k}}(\phi(|\dot q_c(t)|)\dot q_c(t)\cdot\dot\psi(t)
+\nabla V(q_c(t))\cdot\psi(t))\,dt=\int_{B_{k}}(\phi(|\dot q_c(t)|)\dot q_c(t)\cdot\dot\psi(t)
+\nabla V(q_c(t))\cdot\psi(t))\,dt=0.
\]
Hence
%
%
we obtain
\[
\int_{\R}(\phi(|\dot q_c(t)|)\dot q_c(t)\cdot\dot\psi(t)
+\nabla V(q_c(t))\cdot\psi(t))\,dt=0,
\]
so that $q_c$ is a weak solution of \eqref{S} on $\R$.
\end{proof}

We refer to these solutions as \emph{brake--type} solutions. 
They are obtained from $c$--connections which have both contact times finite,
by reflection and periodic continuation. 
These solutions are $T$--periodic with $T=2(\omega-\alpha)$ and symmetric with respect to the two contact times, where they satisfy $q_c(\alpha)\in\mathcal V_c^-$ with $\dot q_c(\alpha)=0$ and $q_c(\omega)\in\mathcal V_c^+$ with $\dot q_c(\omega)=0$. 
This results in a periodic, symmetric back-and-forth oscillation between $\mathcal V_c^-$ and $\mathcal V_c^+$, with velocity tending to zero while reaching these sets. 
They have constant energy $-c$, solve \eqref{S} classically outside the contact times, and are $C^1$ weak solutions of \eqref{S} on $\mathbb R$.

\bigskip

We conclude this section by examining additional asymptotic properties of heteroclinic-- and homoclinic--type orbits. 
We have already established that if a $c$--connection has at most one finite contact time, then the corresponding entire solution $q_c$ is either heteroclinic or homoclinic. 
In both cases, it satisfies
\[
\dist(q_c(t),\mathcal V_c)\to 0 \text{ and } \dot q_c(t)\to 0
\quad \text{as } t\to\pm\infty .
\]
Note moreover that, since by definition $q$ is bounded, we also have by construction that $\|q_c\|_{L^\infty(\R;\R^N)}<+\infty$.

We consider the $\alpha$- and $\omega$-limit sets of $q_{c}$, the set of accumulation points of $q_{c}(t)$ respectively as $t\to-\infty$ and  $t\to+\infty$:\begin{align*}
\mathcal{A}_{q_{c}} &:= \{\xi\in\R^{N}\mid \exists (t_{n})\subset\R \text{ such that }t_{n}\to-\infty\text{ and }\xi=\lim_{n\to+\infty}q_{c}(t_{n})\},\\
\Omega_{q_{c}} &:= \{\xi\in\R^{N}\mid \exists (t_{n})\subset\R \text{ such that }t_{n}\to+\infty\text{ and }\xi=\lim_{n\to+\infty}q_{c}(t_{n})\}.
\end{align*}
They are both closed and connected. Moreover, the boundedness of $q_{c}$ ensures that 
they are nonempty. 

Indeed, assume
$(t_n)\subset\R$ with $t_n \to +\infty$. Since $\|q_c\|_{L^\infty(\R;\R^N)}<+\infty$, along a subsequence (denoted again $(t_{n})$), we have that $q_{c}(t_{n})\to \xi$ for some $\xi\in\R^{N}$. By definition $\xi\in\Omega_{q_{c}}$, hence $\Omega_{q_{c}}\neq\emptyset$. 
By the same reasoning, taking $(t_n)\subset\R$ with $t_n \to -\infty$, we obtain that the sequence $q_{c}(t_{n})$ accumulates at some point $\zeta\in\mathcal{A}_{q_{c}}\neq\emptyset$.\medskip

A further property of the limit sets of $q_{c}$ is that
\begin{equation}\label{eq:V=c}\mathcal{A}_{q_{c}}\cup \Omega_{q_{c}}\subset \overline{B_R(0)}\cap\{\xi\in\R^{N}\mid V(\xi)=c\}.\end{equation}
Indeed, let $\xi\in\Omega_{q_{c}}$. By definition there exists $(t_n)\subset\R$ with $t_n \to +\infty$ such that $q_{c}(t_n) \to \xi$ as $n \to +\infty$. Since $|q_{c}(t_{n})|\leq R$ for any $n\in\N$, it follows that $\xi\in\overline{B_R(0)}$. Moreover, since $q_{c}$ is a heteroclinic or a homoclinic type solution, we know that  $ \dist(q_{c}(t_n), \mathcal{V}_c) \to 0$ as $n\to+\infty$. Hence
$\dist(\xi, \mathcal{V}_c) = 0$, so that $\xi \in \mathcal{V}_c$, since $\mathcal{V}_c$ is closed. Finally, $V(q_c(t_n))\ge c$ for all $n$, and by continuity of $V$ we have $V(q_c(t_n))\to V(\xi)$. Since $\xi\in\mathcal V_c$, we have $V(\xi)\le c$, and therefore $V(\xi)=c$. 
Thus $\Omega_{q_{c}}\subset \overline{B_R(0)}\cap\{\xi\in\R^{N}\mid V(\xi)=c\}$. The same argument applies to $\mathcal A_{q_c}$, proving \eqref{eq:V=c}.\medskip

A final property of the limit sets of $q_{c}$ is that they consist of critical points of $V$
\begin{equation}\label{eq:V'=0}\mathcal{A}_{q_{c}}\cup \Omega_{q_{c}}\subset \{\xi\in\R^{N}\mid \nabla V(\xi)=0\}\end{equation}
Indeed, let $\xi\in\Omega_{q_{c}}$ and $(t_n)\subset\R$ be such that $t_n \to +\infty$ and $q_{c}(t_n) \to \xi$ as $n \to +\infty$. Fix $t_{0}>0$. Since $q_{c}$ satisfies $\dot q_{c}(t)\to 0$ as $t\to\pm\infty$,  we have that
$$
\lim_{n \to +\infty} q_{c}(t_{n}+s) = \lim_{n \to +\infty} q_{c}(t_n) + \lim_{n \to +\infty} \int_{t_n}^{t_{n}+s} \dot{q_{c}}(t) \, dt = \xi\text{ uniformly in  }s\in [0,t_{0}].
$$  
Hence, by continuity,  $\nabla V(q_{c}(t_n + s)) \to \nabla V(\xi)$ uniformly in $s\in [0,t_{0}]$. Since $q_{c}$ is a classical solution of \eqref{S} outside the possible contact point (which is unique in the homoclinic--type case and absent in the heteroclinic--type case), we derive
\begin{equation}\label{eq:inter}
\left(\phi(|\dot{q_{c}}(t_n + s)|)\dot{q_{c}}(t_n + s)\right)'=\nabla V(q_{c}(t_{n}+s)) 
\text{  for any }s\in [0,t_{0}].
\end{equation}
As $\phi(|u|)u\to 0$ for $u\to 0$ and $\dot{q_{c}}(t) \to 0$ as $t \to +\infty$,  we have moreover
\[\phi(|\dot{q_{c}}(t_n + t_{0})|)\dot{q_{c}}(t_n + t_{0}) - \phi(|\dot{q_{c}}(t_n)|)\dot{q_{c}}(t_n)\to 0\text{ as }n\to+\infty.\]
Since
$$
\phi(|\dot{q_{c}}(t_n + t_{0})|)\dot{q_{c}}(t_n + t_{0}) - \phi(|\dot{q_{c}}(t_n)|)\dot{q_{c}}(t_n) = \int_{t_n}^{t_n + t_{0}} \left(\phi(|\dot{q_{c}}(s)|)\dot{q_{c}}(s)\right)' \, ds,
$$  
by integrating both sides of \eqref{eq:inter} with respect to $s$  on the interval $[0,t_{0}]$, we conclude that 
$$t_{0}\nabla V(\xi)=\lim_{n\to+\infty}\int_{t_n}^{t_n + t_{0}} \left(\phi(|\dot{q_{c}}(s)|)\dot{q_{c}}(s)\right)' \, ds=0.$$
Since $t_{0} > 0$, it follows that $\nabla V(\xi) = 0$. The same argument applies to show the inclusion also for 
$\mathcal{A}_{q_{c}}$, proving \eqref{eq:V'=0}.

\begin{remark}\label{R:cregular} As we have seen in this section, for any $c\in\R$ for which \hyperlink{V}{($\mathcal{V}_c$)} and \hyperlink{J}{$(\mathcal{J}_c)$} hold true, there exists an entire connecting solution $q_{c}$ with energy $-c$, which may exhibit different behaviours depending on whether it is of  brake, homoclinic or heteroclinic type.
From \eqref{eq:V=c} and \eqref{eq:V'=0} we deduce that $q_{c}$ can be of heteroclinic or  homoclinic type solution only if $c$ is a critical value of $V$. When $c$ is a regular value of $V$, then $q_{c}$ is necessarily a brake orbit solution.
\end{remark}

\begin{remark}\label{r:regolaritàglobale} As discussed above, the solutions $q_{c}$ may lose $C^{2}$ regularity at the contact
times, where $\dot q_{c}(t)=0$. This happens because the general assumptions
\hyperlink{phi1}{$(\phi_1)$} and \hyperlink{phi2}{$(\phi_2)$} alone are not sufficient
to guarantee that the function $g(\xi)=\phi(|\xi|)\xi$ for $\xi\neq0$, extended by continuity
with $g(0)=0$, is differentiable at $\xi=0$. Consequently, the local invertibility
argument used in the proof of Lemma \ref{l233} does not apply.

This issue can be resolved if $\phi:[0,+\infty)\to(0,+\infty)$ is a $C^{1}$ function
satisfying \hyperlink{phi1}{$(\phi_1)$} and \hyperlink{phi2}{$(\phi_2)$}, together with
the additional assumption $\phi(0)>0$. In that case, the computation in the proof of
Lemma \ref{l233} shows that $g$ is $C^{1}$ also at $\xi=0$ with $Dg(0)=\phi(0)I$,
hence $g$ is locally invertible everywhere. Therefore, when $\phi(0)>0$, the argument
of Lemma \ref{l233} applies to show that the solutions $q_{c}$ are in fact global
$C^{2}$ solutions of \eqref{S} on $\R$. Moreover, the continuation across the contact
times is unique, so $q_c$ is the unique global extension of the given $c$--connection.
\end{remark}
 
%
%
\section{Applications to Classical Potentials}\label{sec:applications} 

Following \cite{AMZ}, we illustrate how the abstract variational
framework developed in Section \ref{Section3} applies to classical
models such as double--well, Duffing--type, and multiple pendulum
potentials. Such potentials are commonly considered in
the variational literature on homoclinic and heteroclinic solutions, see,
e.g., \cite{BolotinKozlov,ambrosetti1992homoclinics,rabinowitz1991results,rabinowitz2000results,alikakos2008connection,alessio2016brake,antonopoulos2016minimizers,sternberg2016heteroclinic}.

We show that these models fit naturally into the abstract framework
introduced in Section \ref{Section3}. In particular, both the separation
condition \hyperlink{V}{$(\mathcal V_c)$} and the coercivity condition \hyperlink{V}{$(\mathcal J_c)$}
hold for all values of $c$ below a natural threshold $c^{*}$ determined
by the geometry of the potential $V$. Hence the existence results
established in Section \ref{Section3} apply to all such values of $c$
and provide connecting trajectories with prescribed energy $-c$ for the
quasilinear system \eqref{S}.

As in \cite{AMZ}, the method also allows one to address multiplicity
issues. This aspect will be illustrated in the case of multiple
pendulum--type systems, where the sublevel set $\mathcal V_c$ admits
different admissible decompositions into well--separated subsets,
leading correspondingly to distinct connecting orbits at the same
prescribed energy level.

%
	
\subsection{Double-well potential systems}

Throughout this subsection, we assume that the potential function $V:\mathbb{R}^N\to\mathbb{R}$ belongs to $C^1(\mathbb{R}^N;\mathbb{R})$ and satisfies the following conditions:

\begin{itemize}
	\item[\hypertarget{V1}{$(V_1)$}] There exist two distinct points $a^-, a^+ \in \mathbb{R}^N$ such that  
	$$
	V(a^+) = V(a^-) = 0, \quad \text{and} \quad V(x) > 0 \quad \text{for all } x \in \mathbb{R}^N \setminus \{a^-, a^+\}.
	$$
	\item[\hypertarget{V2}{$(V_2)$}] The potential is bounded away from zero at infinity, namely 
	$$
	\liminf_{|x| \to +\infty} V(x) =: \nu_0 > 0.
	$$
\end{itemize}

These conditions describe a typical \emph{double--well potential}, where $V$
attains its minimum at two isolated points, is positive elsewhere, and is
uniformly positive at infinity.

Let $\bar r=\tfrac12|a^- - a^+|$ and set
\[
c^{*}:=\inf_{x\in \mathbb R^N\setminus\bigl(B_{\bar r}(a^-)\cup B_{\bar r}(a^+)\bigr)} V(x).
\]

\begin{remark}\label{R:c*double}
By \hyperlink{V1}{$(V_1)$} and \hyperlink{V2}{$(V_2)$} one has
\[
c^{*}\in (0,\nu_{0}].
\]
Indeed $c^{*}\leq \nu_{0}$ by definition of $\nu_{0}$.
To prove that $c^{*}>0$, assume by contradiction that $V(x_n)\to0$ for some
$(x_n)\subset \mathbb R^N\setminus\bigl(B_{\bar r}(a^-)\cup B_{\bar r}(a^+)\bigr)$.
Then $(x_n)$ is bounded by \hyperlink{V2}{$(V_2)$} and, up to a subsequence, $x_n\to x_0$.
Hence $V(x_{0})=0$ while $x_0\neq a^\pm$, contradicting \hyperlink{V1}{$(V_1)$}.
\end{remark}
\begin{lemma}\label{l31}
	If $V \in C^1(\mathbb{R}^N;\mathbb{R})$ satisfies \hyperlink{V1}{$(V_1)$} and \hyperlink{V2}{$(V_2)$}, then 
	\begin{itemize}
	\item[(a)] For every  $c \in [0, c^{*})$,  condition  \hyperlink{V}{$(\mathcal{V}_c)$} holds with respect to two disjoint compact sets  $\mathcal{V}_c^-$ and $\mathcal{V}_c^+$ such that 
		$$
		a^{+}\in\mathcal{V}_c^+\subset B_{\bar r}(a^+)\quad\text{and}\quad  a^{-}\in\mathcal{V}_c^-\subset B_{\bar r}(a^-).
		$$ 
		In particular, in the case $c = 0$, $\mathcal{V}_c^- = \{a^-\}$ and $\mathcal{V}_c^+ = \{a^+\}$.

		\item[(b)] For every $c \in [0, c^{*})$, the coercivity condition \hyperlink{J}{$(\mathcal{J}_c)$} holds. \vspace{0.2cm}
			\end{itemize}
\end{lemma}	
\begin{proof}
\noindent {\it (a)} 
If $c<c^{*}$, then $V> c$ on
$\mathbb R^N\setminus\bigl(B_{\bar r}(a^-)\cup B_{\bar r}(a^+)\bigr)$, and therefore
\[
\mathcal V_c\subset B_{\bar r}(a^-)\cup B_{\bar r}(a^+).
\]
Set
\[
\mathcal V_c^-:=\mathcal V_c\cap B_{\bar r}(a^-),\qquad
\mathcal V_c^+:=\mathcal V_c\cap B_{\bar r}(a^+).
\]
Then $\mathcal V_c^\pm$ are compact and disjoint, and satisfy
$\mathcal V_c=\mathcal V_c^-\cup \mathcal V_c^+$. For $c=0$, by \hyperlink{V1}{$(V_{1})$} one has
$\mathcal V_0^-=\{a^-\}$ and $\mathcal V_0^+=\{a^+\}$.

{\it (b)} Fix $c\in[0,c^{*})$.	Since $c<c^{*}\leq \nu_{0}$ we have $\mu=(\nu_{0}-c)/2>0$, and
 there exists $R_{0}>0$ such that
$
V(x)\ge c+\mu$ for all $|x|\ge R_0$.
Suppose by contradiction that \hyperlink{J}{$(\mathcal J_c)$} fails.
Then there exists $(q_n)\subset\Gamma_c$ such that
$
J_c(q_n)\to m_c$ and
$\|q_n\|_{L^\infty(\mathbb R;\mathbb R^N)}\to+\infty$.
Since $\|q_n\|_{L^\infty(\mathbb R;\mathbb R^N)}\to+\infty$, for any $R>R_0$ and $n$ large,  there exist $s_n<t_n$ such that
\[
|q_n(t_n)-q_n(s_n)|\ge R-R_0
\quad\text{and}\quad
V(q_n(t))\ge c+\mu \quad\text{for all } t\in(s_n,t_n).
\]
By Lemma~\ref{l4}, there exists $\lambda_\mu>0$ such that
$
J_c(q_n)\ge J_{c,(s_n,t_n)}(q_n)
\ge \lambda_\mu\, H(R-R_0).
$
Since $H(r)\to+\infty$ as $r\to+\infty$, letting $R\to+\infty$ yields
$m_c=+\infty$, a contradiction. Hence \hyperlink{J}{$(\mathcal J_c)$} holds.
\end{proof}

The previous lemma shows that, for every $c<c^*$, the potential $V$
satisfies the separation condition \hyperlink{V}{$(\mathcal V_c)$} and the coercivity
condition \hyperlink{J}{$(\mathcal J_c)$}. Hence Theorem \ref{thm:minimal-connection} applies.

\begin{proposition}\label{p:doublewell-connection}
Assume that $V\in C^{1}(\R^{N};\R)$ satisfies \hyperlink{V1}{$(V_{1})$} and \hyperlink{V2}{$(V_{2})$}.
Then for every $c\in[0,c^{*})$ there exist
$I=(\alpha,\omega)$ with $-\infty\le \alpha<\omega\le +\infty$
and a minimal $c$--connection
$q:\bar I\to\R^{N}$ between $\mathcal V_c^-$ and $\mathcal V_c^+$.
In particular,
$
q\in C^{2}(I;\R^{N})\cap C^{1}(\bar I;\R^{N}),
$
and solves \eqref{S} on $I$ with
$
E_q(t)=-c$ for all $t\in I .
$
\end{proposition}

By the extension results of Subsection \ref{Subsec3.3},
this connection generates an entire weak solution $q_c$ of \eqref{S}
with constant energy $-c$.

By Remark~\ref{R:cregular}, if $c$ is a regular value of $V$,
then the extended weak solution $q_c$ is necessarily of brake type.
%


\subsection{Duffing-like systems}
	
In this subsection we assume that the potential $V:\mathbb{R}^N\to\mathbb{R}$  belongs to $C^1(\mathbb{R}^N;\mathbb{R})$ and satifies
\begin{itemize}
	\item[\hypertarget{V3}{($V_3$)}] There exists $r_0>0$ such that
	$$ 
	V(x) > 0\quad \text{for any }~ x \in B_{4r_0}(0) \setminus \{0\},
	$$ 
	that is, $V$ has a strict local minimum at $x_0 := 0$ with value $V(0) = 0$.
	\item[\hypertarget{V4}{($V_4$)}] The set 
	$$
	C_0 := \left\{x \in \mathbb{R}^N : V(x) > 0\right\} \cup \{0\},
	$$ 
	is bounded and satisfies $\nabla V(x) \neq 0$ for any $x \in \partial C_0$.
\end{itemize}
These assumptions describe a typical \emph{Duffing--type potential}, in which
$V$ has a strict isolated minimum at the origin,
is positive in a bounded region around it,
and changes sign across the boundary of this region.
\begin{remark}\label{R:C0-geometry}
Since $V\in C^{1}(\mathbb R^{N};\R)$ and $\nabla V(x)\neq0$ for every
$x\in\partial C_0$, the set $\partial C_0$ is a compact $C^{1}$ hypersurface.
Moreover, denoting
$
\nu_x=-\frac{\nabla V(x)}{\|\nabla V(x)\|}
$
the outward unit normal, compactness implies that there exists
$v_0>0$ such that
$
\nabla V(x)\cdot\nu_x=-\|\nabla V(x)\|\le -v_0$
for all $x\in\partial C_0$.
As a consequence, there exists $\delta_0>0$ such that
\begin{equation}\label{eq:vicinoC0}V(x)<0\text{ whenever
}0<\dist(x,\overline C_{0})<\delta_0.\end{equation}
\end{remark}
As in the double--well case, the geometry of the potential $V$
naturally determines a threshold value $c^*$.
We set 
$\bar r=\frac 12\dist(0,\R^{N}\setminus C_{0})$,
noting that by \hyperlink{V3}{$(V_3)$} one has
$B_{4r_{0}}(0)\subset C_{0}$ and therefore $\bar r\geq 2r_{0}>0$.
For a given $A\subset\mathbb R^{N}$ and $r>0$ we denote
\[
N_{r}(A)
=\{x\in\mathbb R^N: \dist(x,A)<r\}.\]
We define
\[
c^{*}
=
\inf\{V(x):
x\in C_0\setminus \bigl(B_{\bar r}(0)\cup N_{\bar r}(\mathbb R^N\setminus C_0)\bigr)
\}.
\]
This is the natural analogue of the definition of $c^{*}$ in the double--well case,
with the isolated minimum $0$ and the boundary $\partial C_0$
playing the role of the two wells.

Observe that $c^{*}>0$. Indeed, assume by contradiction $c^{*}=0$. Then there exists a sequence $(x_{n})\subset C_{0}$ with $|x_{n}|\ge \bar r$, $\dist(x_{n},\mathbb R^N\setminus C_0)\ge \bar r$ and $V(x_{n})\to 0$. Since $C_{0}$ is bounded, up to a subsequence we have $x_{n}\to x_{0}$. By continuity of $V$ we have $V(x_{0})=0$, while $x_{0}\in C_{0}\setminus\{0\}$, a contradiction because by definition of $C_{0}$ one has $V(x)>0$ on $C_{0}\setminus\{0\}$

\begin{lemma}\label{l321}
If $V \in C^1(\mathbb{R}^N;\mathbb{R})$ satisfies \hyperlink{V3}{$(V_3)$} and \hyperlink{V4}{$(V_4)$}, then
	\begin{itemize}
	\item[(a)] For every $c \in [0, c^*)$, the condition \hyperlink{V}{$(\mathcal{V}_c)$} is satisfied with respect to two disjoint closed sets  $\mathcal{V}_c^-$ and $\mathcal{V}_c^+$ such that 
		$$
		\partial C_{0}\subset \mathcal{V}_c^+\subset \R^{N}\setminus \overline B_{\bar r}(0)\quad\text{and}\quad  0\in \mathcal{V}_c^-\subset B_{\bar r}(0).
		$$ 
		\item[(b)] For every $c \in [0,c^{*})$, the coercivity condition \hyperlink{J}{$(\mathcal{J}_c)$} holds.
		
	\end{itemize}
\end{lemma}
\begin{proof} \noindent{\it (a)} Fix $c\in[0,c^*)$. By definition of $c^*$ one has
\[
V(x)\ge c^*>c \quad \text{for all }x\in C_0\setminus (B_{\bar r}(0)\cup N_{\bar r}(\R^N\setminus C_0)).
\]
Since moreover $V(x)\le 0$ for $x\notin C_0$, it follows that
\[
\mathcal V_c \subset B_{\bar r}(0)\cup N_{\bar r}(\mathbb R^N\setminus C_0).
\]
Set
\[
\mathcal V_c^-=\mathcal V_c\cap B_{\bar r}(0)\text{ and }\mathcal V_c^+=\mathcal V_c\cap N_{\bar r}(\mathbb R^N\setminus C_0).
\]
Then $\mathcal V_{c}=\mathcal V_c^-\cup \mathcal V_c^+$ and
$\mathcal V_c^-\cap\mathcal V_c^+=\emptyset$. Moreover $0\in\mathcal V_c^-$ and $\partial C_0\subset \mathcal V_c^+$ since $V=0$ on $\partial C_{0}$.
Finally $\mathcal V_c^-$ is compact and $\mathcal V_c^+$ is closed. Hence, 
\[
\dist(\mathcal V_c^-,\mathcal V_c^+)>0,
\]
and therefore condition \hyperlink{V}{$(\mathcal V_c)$} holds.

\noindent{\it (b)} Fix $c\in[0,c^{*})$ and let $q\in\Gamma_{c}$.
Then by definition $V(q(t))\ge c$ for all $t\in\R$. 

If $c>0$, this implies $V(q(t))>0$ for all $t$, hence
$q(t)\in C_{0}$ for every $t\in\R$.
Since $C_0$ is bounded, it follows that $q$ is bounded.

When $c=0$ we claim that $q(t)\in\overline C_0$ for all $t\in\R$. Again
since $C_0$ is bounded, this implies that $q$ is bounded and (b) follows.

To prove the claim, argue by contradiction. Assume that there exists $t_1\in\R$ such that
$\dist(q(t_1),\overline C_0)>0$. Since $\mathcal V_c^-\subset B_{\bar r}(0)\subset C_0$, and 
$\liminf_{t\to-\infty}\dist(q(t),\mathcal V_c^-)=0$,
there exists $t_0$ such that $q(t_0)\in C_0$ so that $\dist(q(t_0),\overline C_0)=0$. By continuity of the map $t\mapsto \dist(q(t),\overline C_0)$, there is a point $t_{2}$,  between $t_{0}$ and $t_{1}$, such that
$\delta_{0}>\dist(q(t_2),\overline C_0)>0$. By \eqref{eq:vicinoC0} this yields $V(q(t_2))<0$,
contradicting $V(q(t))\ge c=0$.
\end{proof}

By Lemma \ref{l321}, Theorem \ref{thm:minimal-connection} applies.
\begin{proposition}\label{p:duffing-connection}
Assume that $V\in C^{1}(\R^{N};\R)$ satisfies \hyperlink{V3}{$(V_{3})$} and \hyperlink{V4}{$(V_{4})$}.
Then for every $c\in[0,c^{*})$ there exist
$I=(\alpha,\omega)$ with $-\infty\le \alpha<\omega\le +\infty$
and a minimal $c$--connection
$q:\bar I\to\R^{N}$ between $\mathcal V_c^-$ and $\mathcal V_c^+$. In particular,
$
q\in C^{2}(I;\R^{N})\cap C^{1}(\bar I;\R^{N}),
$
solves \eqref{S} on $I$, and
$E_q(t)=-c$ for all $t\in I$.
\end{proposition}

We remark again that each minimal $c$--connection given by Proposition \ref{p:duffing-connection} generates
an entire weak solution $q_c$ of \eqref{S} with constant energy $-c$.


\subsection{Multiple pendulum-type systems}

We now finally consider the case of \textit{multiple pendulum-type systems} where the potential $V\in C^1(\mathbb{R}^N;\mathbb{R})$ satisfies:

\begin{itemize}
	\item[\hypertarget{V5}{$(V_5)$}] $V$ is $\mathbb{Z}^N$-periodic, namely
	$
	V(x + \xi) = V(x)$ for all $x \in \mathbb{R}^N$ and $\xi \in \mathbb{Z}^N.$
	\item[\hypertarget{V6}{$(V_6)$}] $V(x) \geq 0$ for all $x \in \mathbb{R}^N$, and $V(x) = 0$ if and only if $x \in \mathbb{Z}^N$.
\end{itemize}\bigskip	
Assumptions \hyperlink{V5}{$(V_5)$}--\hyperlink{V6}{$(V_6)$} induce a simple
geometric structure on the configuration space $\mathbb R^N$, which allows us,
as in the preceding cases, to identify a threshold value $c^{*}$ below which the
sublevel set $\mathcal V_c$ enjoys suitable separation properties.
More precisely, we set
\[
c^{*}:=\inf\{V(x):\dist(x,\mathbb Z^N)\ge 1/2\}.
\]
By \hyperlink{V6}{$(V_6)$}, $V$ is strictly positive on
$[0,1]^N\setminus N_{1/2}(\mathbb Z^N)$.
Since this set is compact and $V$ is continuous, it follows that
$c^{*}>0$.

Fix $c\in[0,c^{*})$ and set $\nu_c=(c^{*}-c)/2$.
By continuity and $\mathbb Z^N$--periodicity of $V$, there exists
$\delta_c\in(0,1/2)$ such that
\begin{equation}\label{eq:deltac}
V(x)\ge c+\nu_c
\text{ whenever }
\dist(x,\mathbb Z^N)>1/2-\delta_c.
\end{equation}
Consequently,
\[
\mathcal V_c\subset\{x:\dist(x,\mathbb Z^N)\le 1/2-\delta_c\}
=\cup_{\xi\in\mathbb Z^N}\overline B_{1/2-\delta_c}(\xi).
\]
Hence, defining
\[
\mathcal V_{c,\xi}=\mathcal V_c\cap \overline B_{1/2-\delta_c}(\xi),
\]
it follows that each $\mathcal V_{c,\xi}\neq\emptyset$ is closed, and that
\[
\mathcal V_c=\cup_{\xi\in\mathbb Z^N}\mathcal V_{c,\xi}\text{ and }
\dist(\mathcal V_{c,\xi},\mathcal V_{c,\xi'})\ge 2\delta_c
\quad\text{for all }\xi\neq\xi'\in\mathbb Z^N.
\]
In particular we have

\begin{lemma}\label{lZ}
Assume that $V\in C^{1}(\mathbb R^{N};\mathbb R)$ satisfies
\hyperlink{V5}{$(V_5)$}--\hyperlink{V6}{$(V_6)$}. Let $Z_\pm\subset\mathbb Z^{N}$ be nonempty sets such that
$Z_-\cap Z_+=\emptyset$ and $\mathbb Z^{N}=Z_-\cup Z_+$. Then, for every $c\in[0,c^{*})$, setting
\[
\mathcal V_c^-(Z_{-})= \cup_{\xi\in Z_-} \mathcal V_{c,\xi}\text{ and }
\mathcal V_c^+(Z_{+})= \cup_{\xi\in Z_+} \mathcal V_{c,\xi},
\]
we have 
\[\mathcal V_c=\mathcal V_c^-(Z_{-})\cup\mathcal V_c^+(Z_{+}),\ Z_\pm\subset \mathcal V_c^\pm(Z_{\pm})\text{ and 
}\dist(\mathcal V_c^-(Z_{-}),\mathcal V_c^+(Z_{+}))\geq 2\delta_{c}.\]
In particular the sets $\mathcal V_c^-(Z_{-})$ and $\mathcal V_c^+(Z_{+})$ are closed.
\end{lemma}
 \begin{proof}
Fix $c\in[0,c^*)$ and let $\mathbb Z^N=Z_-\cup Z_+$ be a disjoint partition as in the statement. Since $\mathcal V_c=\cup_{\xi\in Z^N}\mathcal V_{c,\xi}$ and $Z_{-}\cup Z_{+}=\mathbb Z^{N}$, it follows that
$
\mathcal V_c=\mathcal V_c^-(Z_{-})\cup\mathcal V_c^+(Z_{+}).
$
Since $\xi\in\mathcal V_{c,\xi}$ for any $\xi\in\mathbb Z^N$ we obtain also
$Z_\pm\subset\mathcal V_c^\pm(Z_{\pm})$.
Finally, if $x\in\mathcal V_c^-(Z_{-})$ and $y\in\mathcal V_c^+(Z_{+})$, then there exist
$\xi\in Z_-$ and $\xi'\in Z_+$ such that $x\in\mathcal V_{c,\xi}$ and
$y\in\mathcal V_{c,\xi'}$. Hence, since $\xi\neq\xi'$, 
$
|x-y|\ge \dist(\mathcal V_{c,\xi},\mathcal V_{c,\xi'})\ge 2\delta_c$
and therefore $\dist(\mathcal V_c^-(Z_{-}),\mathcal V_c^+(Z_{+}))\ge 2\delta_c$.
Consequently, since $\mathcal V_c=\mathcal V_c^-(Z_{-})\cup\mathcal V_c^+(Z_{+})$ is closed and $\dist(\mathcal V_c^-(Z_{-}),\mathcal V_c^+(Z_{+}))>0$, the 
 sets $\mathcal V_c^\pm(Z_{\pm})$ are closed too.
\end{proof}

Lemma~\ref{lZ} provides a natural decomposition of the sublevel set
$\mathcal V_c$ for every $c\in[0,c^*)$, associated with an arbitrary
partition $\mathbb Z^N=Z_-\cup Z_+$. In particular, this decomposition
ensures that assumption \hyperlink{V}{$(\mathcal V_c)$} holds for all
$c\in[0,c^*)$.
In order to apply the abstract variational results developed in the previous
sections and obtain the existence of connecting orbits, it therefore remains
to verify the coercivity assumption \hyperlink{J}{$(\mathcal J_c)$}
for suitable choices of the sets $Z_\pm$.

The key ingredient for establishing \hyperlink{J}{$(\mathcal J_c)$}
is the following estimate.

\begin{lemma}\label{R:1}
Let $c\in[0,c^{*})$ and $M>0$. Then there exists $R_{0}=R_{0}(c,M)>0$
such that the following holds.  
If $q\in W^{1,\Phi}_{\mathrm{loc}}(\R;\R^{N})$ and $\sigma<\tau$ satisfy
\[
V(q(t))\ge c \quad \text{for all } t\in(\sigma,\tau),
\qquad
J_{c,(\sigma,\tau)}(q)\le M,
\]
then
\[
|q(\tau)-q(\sigma)|\le R_{0}.
\]
\end{lemma}
\begin{proof}
Let $M>0$. Arguing by contradiction, assume that there exist sequences
$(q_{n}) \subset W^{1, \Phi}_{\mathrm{loc}}(\mathbb{R};\mathbb{R}^N)$ and 
    $\sigma_{n}<\tau_{n}\in \mathbb{R}$ satisfying 
    $
    V(q_{n}(t)) \geq c$ on $(\sigma_{n},\tau_{n})$, $ 
    J_{c, (\sigma_{n}, \tau_{n})}(q_{n}) \leq M$
    and
    \[
    |q_{n}(\tau_{n}) - q_{n}(\sigma_{n})| \to +\infty.
    \]
Since $|q_{n}(\tau_{n}) - q_{n}(\sigma_{n})| \to +\infty$, there exists an index
$i\in \{1,\ldots, N\}$ such that
\begin{equation}\label{eq:componenteiinfty}
|q_{n}(\tau_{n})_{i} - q_{n}(\sigma_{n})_{i}| \to +\infty.
\end{equation}
For such value of $i$, define the slabs
\[S_{i,j}=\{x\in\mathbb R^{N}\mid 1/2-\delta_{c}<x_{i}-j<1/2+\delta_{c}\},\quad j\in\mathbb Z.
\]
By definition, $S_{i,j}\subset\mathbb R^{N}\setminus\cup_{\xi\in\mathbb Z^N}\overline B_{1/2-\delta_c}(\xi)$. Hence, by \eqref{eq:deltac}, $V(x)\geq c+\nu_{c}$ for any $x\in S_{i,j}$ and $j\in\mathbb Z$. By \eqref{eq:componenteiinfty} and the continuity of $q_{n}$, it follows that, as $n\to\infty$, the image
$q_{n}((\sigma_{n},\tau_{n}))$ crosses an increasing number of distinct slabs. More precisely, there exist integers $j^{-}_n<j^{+}_n\in\mathbb Z$ with
$j^{+}_n-j^{-}_n\to+\infty$ and times
\[
\sigma_n<\alpha_{j^{-}_n}<\beta_{j^{-}_n}<\alpha_{j^{-}_n+1}<\beta_{j^{-}_n+1}
<\ldots<\alpha_{j^{+}_n}<\beta_{j^{+}_n}<\tau_n
\]
such that, for every $j\in\mathbb Z$ with $j^{-}_n\le j\le j^{+}_n$, one has
\[
(q_n(\alpha_j))_i-j=\tfrac12-\delta_{c},\quad
(q_n(\beta_j))_i-j=\tfrac12+\delta_{c},
\text{
and
}
q_n(t)\in S_{i,j}\quad\text{for all }t\in(\alpha_j,\beta_j).
\]
Then, 
$V(q_{n}(t)) -c\geq \nu_{c}$ for every $t\in (\alpha_j, \beta_j)$ and $j\in [j^{-}_{n},j^{+}_{n}]\cap\mathbb Z$.
Since $|q_n(\beta_j)-q_n(\alpha_j)|\geq 2\delta_{c}$  for any $j\in [j^{-}_{n},j^{+}_{n}]\cap\mathbb Z$, by Lemma \ref{l4}, we obtain 
$$
J_{c, (\alpha_j, \beta_j)}(q_{n}) \geq \lambda_{\nu_{c}}H(2\delta_{c})>0.
$$
Hence \[M\geq J_{c, (\sigma_n, \tau_n)}(q_{n})\geq \sum_{j=j^{-}_{n}}^{j^{+}_{n}} J_{c, (\alpha_j, \beta_j)}(q_{n})\geq (j^{+}_{n}-j^{-}_{n}+1)\lambda_{\nu_{c}}H(2\delta_{c})\quad\forall n\in\N,\]
contradicting $j^{+}_{n}-j^{-}_{n}\to+\infty$ and concluding the proof. 
\end{proof}

Given a partition $\mathbb Z^{N}=Z_-\cup Z_+$ and the associated sets $\mathcal V_c^-(Z_{-})$ and $\mathcal V_c^+(Z_{+})$ as in 
Lemma ~\ref{lZ}, in the following we denote
\[\Gamma_{c}(Z_{-},Z_{+})=\{ q \in \mathcal{X}_c \mid \liminf_{t \to -\infty} \operatorname{dist}(q(t), \mathcal{V}_c^-(Z_{-})) = \liminf_{t \to +\infty} \operatorname{dist}(q(t), \mathcal{V}_c^+(Z_{+})) = 0 \}.\]

Thanks to Lemma~\ref{lZ} and Lemma~\ref{R:1},  \hyperlink{V}{$(\mathcal{V}_c)$} and \hyperlink{J}{$(\mathcal J_c)$} hold for many different choices of the sets $Z_\pm$. 
This makes it possible to obtain multiple geometrically distinct connecting orbits.
To exploit this fact, we proceed inductively, 
adapting the arguments in~\cite{BolotinKozlov, [BM]}, and repeatedly using the following property.
\begin{lemma}\label{L:decompZpm} For any $c\in [0,c^*)$ there exists $R=R(c)>0$ such that, for every partition $\mathbb Z^N=Z_-\cup Z_+$ satisfying  
\begin{equation}\label{eq:Z-}0\in Z_-\neq\mathbb Z^{N}\text{ and }\xi+Z_-\subset Z_{-}\text{ for any }\xi\in Z_{-},\end{equation} 
letting $\mathcal V_{c}^{\pm}(Z_{\pm})$ be the the sets given by Lemma~\ref{lZ}, the functional $J_c$ satisfies \hyperlink{J}{($\mathcal{J}_c$)} on $\Gamma_{c}(Z_{-},Z_{+})$ 
with $R=R(c)$, namely, 
$$ \inf_{q \in \Gamma_c(Z_{-},Z_{+})} J_c(q) = \inf \left\{ J_c(q) : q \in \Gamma_c(Z_{-},Z_{+}) \text{ and } \|q\|_{L^{\infty}(\mathbb{R}; \mathbb{R}^N)} \leq R \right\}. $$
\end{lemma}
\begin{proof} Let $
m_c=\inf_{\Gamma_c(Z_{-},Z_{+})} J_c
$. We claim that
\begin{equation}\label{eq:boundmc}m_{c}\le \Phi(1)+\displaystyle\max_{|x|\leq 1} V(x):=M.\end{equation}
To prove \eqref{eq:boundmc}, denoting by $e_m$ the $m$-th canonical basis vector of $\mathbb R^N$, we observe
that, since $Z_-\subset\mathbb Z^N$ is invariant under addition and
$Z_-\neq\mathbb Z^N$, it cannot contain both $e_m$ and $-e_m$ for every
$m\in\{1,\ldots,N\}$. Indeed, if this were the case, then by closure under
addition one would have $Z_-=\mathbb Z^N$, a contradiction.
Hence there exists $m\in\{1,\ldots,N\}$ and $\xi_{+}\in\{e_{m},-e_{m}\}$ such that $\xi_{+}\in Z_+$.

Consider the interpolating function $\zeta$, defined by 
$
\zeta(t) = t \xi_{+}$ if $t \in [0,1]$, $\zeta(t)=0$ if $t \leq 0$, and $\zeta(t)=\xi_{+}$ if $t \geq 1$.
Set
$$\sigma=\sup\{t\in [0,1]\mid \zeta(t)\in \mathcal V_{c,0}\},\quad \tau=\inf\{t\in (\sigma,1]\mid \zeta(t)\in \mathcal V_{c,\xi_{+}}\}.$$
By continuity $\zeta(\sigma)\in  \mathcal V_{c,0}$,  $\zeta(\tau)\in  \mathcal V_{c,\xi_{+}}$ and $\zeta(t)\notin \mathcal V_{c,0}\cup \mathcal V_{c,\xi_{+}}$ for any $t\in (\sigma,\tau)$. Hence, 
because by construction $\zeta([0,1])\subset  \bar B_{1/2}(0)\cup \bar B_{1/2}(\xi_{+})$ and $\mathcal V_{c}\setminus (\mathcal{V}_{c,0}\cup \mathcal{V}_{c,\xi_{+}})\subset \R^{N}\setminus (\bar B_{1/2}(0)\cup \bar B_{1/2}(\xi_{+}))$, we obtain that $\zeta(t)\in\R^{N}\setminus \mathcal V_{c}$, and so,   
$
V(\zeta(t)) > c$, for all $t \in (\sigma,\tau).$

Defining the function
$$
\eta(t) =
\begin{cases}	
	\zeta(\sigma), & t\leq \sigma, \\
	\zeta(t), & t \in (\sigma, \tau), \\
	\zeta(\tau), & t \geq \tau,
\end{cases}
$$
we have $\eta\in \Gamma_{c}(Z_{-},Z_{+})$, and since $|\dot \eta|\leq 1$, it follows $
m_{c}\le J_c(\eta)\leq \Phi(1)+\displaystyle\max_{|x|\leq 1} V(x)
$,  as we claimed.  

Let $(q_n)\subset \Gamma_c(Z_{-},Z_{+})$ be a minimizing sequence. By \eqref{eq:boundmc} we can assume that $J_c(q_{n})\leq M+1$ for all $n$. 
By definition of $\Gamma_c(Z_{-},Z_{+})$, there exist sequences $t_n^-\to -\infty$ and $\xi_n^-\in Z_-$ such that
\[
\lim_{n\to+\infty} \dist(q_n(t_n^-),\mathcal V_c^-(Z_{-}))=
\lim_{n\to+\infty} \dist(q_n(t_n^-),\mathcal V_{c,\xi_n^-})=0.
\]
Since $\mathcal V_{c,\xi_n^-}\subset B_{1/2}(\xi_n^-)$ we can choose the sequence $(t_n^-)$ such that 
\[|q_n(t_n^-)-\xi_{n}^{-}|\leq 1/2\text{ for all }n\in\N.\]
Consider the translated sequence 
\[
v_n(t)=q_n(t)-\xi_n^-.
\]
By $\mathbb Z^{N}$-periodicity of $V$, we have $V(v_n(t))\geq c$ for all $t\in\mathbb R$ and $J_c(v_n)=J_c(q_n)\leq M+1$. 
Moreover, since $\mathcal V_{c,0}=\mathcal V_{c,\xi_n^-}-\xi_n^-$, we obtain
\begin{equation}
\label{eq:distvn}
\dist\bigl(v_n(t_n^-),\mathcal V_{c,0}\bigr)=
\dist\bigl(q_n(t_n^-),\mathcal V_{c,\xi_n^-}\bigr)\to 0\text{ and }|v_n(t_n^-)|\leq 1/2\text{ for all }n\in\N.
\end{equation}
Since $J_c(v_n)\leq M+1$ and $V(v_n(t))\geq c$ for all $t$, Lemma \ref{R:1} yields a constant $R_{0}=R_{0}(c,M+1)$ such that
$
|v_n(t)-v_n(t_n^-)|\leq R_{0}$ for all $n\in\mathbb N$ and $t\in\mathbb R$. Therefore, setting $R(c)=1/2+ R_{0}$ we obtain
\begin{equation}\label{eq:boundinfty}
	\|v_{n}\|_{L^\infty(\R;\R^N)}\leq \sup_{n\in\N}|v_n(t_n^-)|+R_{0}\leq R(c)\text{ for all }n\in\N.
\end{equation}
To conclude the proof of the Lemma, by \eqref{eq:boundinfty} it is sufficient to show that $(v_{n})$ is a minimizing sequence of $J_{c}$ on $\Gamma_{c}(Z_{-},Z_{+})$. To this aim, we first recall that  $J_c(v_n)=J_c(q_n)\to m_{c}$. Secondly, by \eqref{eq:distvn} we  have also $\liminf_{t\to -\infty} \dist(v_n(t),\mathcal V_c^-(Z_{-}))=0$. Finally, since $q_{n}\in \Gamma_{c}(Z_{-},Z_{+})$, there  exist sequences $t_n^+\to +\infty$ and $\xi_n^+\in Z_+$ such that
\[
\lim_{n\to+\infty} \dist(q_n(t_n^+),\mathcal V_c^+(Z_{+}))=
\lim_{n\to+\infty} \dist(q_n(t_n^+),\mathcal V_{c,\xi_n^+})=0.
\]
Since $v_{n}=q_{n}-\xi_{n}^{-}$, we obtain
$
\lim_{n\to+\infty} \dist(v_n(t_n^+),\mathcal V_{c,\xi_n^+-\xi_{n}^{-}})=0.$ Then, noting that $\xi_n^+-\xi_{n}^{-}\in Z_{+}$ (since $\xi_{n}^{-}+Z_{-}\subset Z_{-}$), we conclude
\[\lim_{n\to+\infty} \dist(v_n(t_n^+),\mathcal V_{c}^{+}(Z_{+}))=\lim_{n\to+\infty} \dist(v_n(t_n^+),\cup_{\xi\in Z_{+}}\mathcal V_{c,\xi})\leq\dist(v_n(t_n^+),\mathcal V_{c,\xi_n^+-\xi_{n}^{-}})=0,\]
showing that $\liminf_{t\to+\infty} \dist(v_n(t),\mathcal V_{c}^{+}(Z_{+}))=0$, and so that $v_{n}\in\Gamma_{c}(Z_{-},Z_{+})$.	
\end{proof}

Lemma~\ref{lZ} provides, for every $c\in[0,c^*)$, a decomposition
$\mathcal V_c=\mathcal V_c^-(Z_{-})\cup\mathcal V_c^+(Z_{+})$ associated with any partition
$\mathbb Z^N=Z_-\cup Z_+$, and ensures the separation condition \hyperlink{V}{$(\mathcal V_c)$}.
Moreover, if the partition is admissible in the sense of \eqref{eq:Z-},
then Lemma~\ref{R:1} and Lemma~\ref{L:decompZpm} yield the coercivity condition
\hyperlink{J}{$(\mathcal J_c)$} on $\Gamma_c(Z_{-},Z_{+})$ with a constant $R(c)$ depending only on $c$.
Hence Theorem~\ref{thm:minimal-connection} applies.

\begin{lemma}\label{R:2}
Let $c\in[0,c^*)$. For any $Z_{-}\subset\mathbb Z^N$ satisfying \eqref{eq:Z-},
set $Z_+:=\mathbb Z^N\setminus Z_-$ and let
$\mathcal V_c^\pm(Z_{\pm})$ be the sets
given by Lemma~\ref{lZ}. Then there exist
$\xi^{+}\in Z_{+}$, an interval $I=(\alpha,\omega)\subset\R$ with
$-\infty\le\alpha<\omega\le+\infty$, and a minimal $c$--connection
$q:\bar I\to\R^{N}$ between $\mathcal V_c^-(Z_{-})$ and $\mathcal V_c^+(Z_{+})$
 such that
\[
\lim_{t\to\alpha^+}\dist\bigl(q(t),\mathcal V_{c,0}\bigr)=0,\quad
\lim_{t\to\omega^-}\dist\bigl(q(t),\mathcal V_{c,\xi^{+}}\bigr)=0,\quad
\|q\|_{L^{\infty}(I;\mathbb R^{N})}\leq 2\,R(c).
\]
\end{lemma}
\begin{proof}
Fix $c\in [0,c^{*})$. Since $Z_{-}$ satisfies \eqref{eq:Z-}, the partition $\mathbb Z^{N}=Z_{-}\cup Z_{+}$ satisfies the assumptions of Lemma \ref{lZ} and Lemma \ref{L:decompZpm}. In particular Lemma \ref{lZ} guarantees that
the property \hyperlink{V}{$(\mathcal V_c)$} holds for the decomposition $\mathcal V_{c}=\mathcal V_{c}^{-}(Z_{-})\cup \mathcal V_{c}^{+}(Z_{+})$. Lemma \ref{L:decompZpm} tells us moreover that $J_{c}$ satisfies \hyperlink{J}{$(\mathcal J_c)$} on $\Gamma_c(Z_{-},Z_{+})$ with $R=R(c)$.
By Theorem \ref{thm:minimal-connection} there exists  an interval
$I=(\alpha,\omega)\subset\R$ with $-\infty\le\alpha<\omega\le+\infty$, and
a minimal $c$--connection  $q_{0}:\bar I\to\R^{N}$
 between $\mathcal V_c^-(Z_{-})$ and $\mathcal V_c^+(Z_{+})$ with $\|q_{0}\|_{L^{\infty}(I;\R^{N})}\leq R(c)$. In particular
\begin{equation}\label{eq:limit+-}\lim_{t\to\alpha^{+}}\dist(q_{0}(t),\mathcal{V}_c^{-}(Z_{-}))=0\text{ and }\lim_{t\to\omega^{-}}\dist(q_{0}(t),\mathcal{V}_c^{+}(Z_{+}))=0.\end{equation}  
By \eqref{eq:limit+-} and the continuity of $q_{0}$, 
since $\mathcal{V}_c^{\pm}(Z_{\pm})$ are unions of the pairwise well separated sets $\{\mathcal{V}_{c,\xi}\}_{\xi\in\mathbb Z^{N}}$, there exist $\zeta^{-}\in Z_{-}$ and $\zeta^{+}\in Z_{+}$ for which 
\begin{equation}\label{eq:limit+-2}\lim_{t\to\alpha^{+}}\dist(q_{0}(t),\mathcal{V}_{c,\zeta^{-}})=0\text{ and }\lim_{t\to\omega^{-}}\dist(q_{0}(t),\mathcal{V}_{c,\zeta^{+}})=0.\end{equation}
Denoting $\xi^{+}=\zeta^{+}-\zeta^{-}$ and $q(t)=q_{0}(t)-\zeta^{-}$ we obtain
that $\|q\|_{L^{\infty}(I;\R^{N})}\leq 2R(c)$,  and, by $\mathbb Z^{N}$ invariance of the problem, $q$ is a minimal $c$-connection
satisfying moreover, by  \eqref{eq:limit+-2}, $
\lim_{t\to\alpha^+}\dist(q(t),\mathcal V_{c,0})=
\lim_{t\to\omega^-}\dist(q(t),\mathcal V_{c,\xi^{+}})=0$.
\end{proof}

We can now state and prove the following multiplicity result.
\begin{proposition}\label{R:3}
Assume that the conditions \hyperlink{phi1}{$(\phi_1)$}--\hyperlink{phi2}{$(\phi_2)$} hold, and let $V\in C^1(\mathbb{R}^N;\mathbb{R})$ satisfy \hyperlink{V5}{$(V_5)$}--\hyperlink{V6}{$(V_6)$}. Then, for every $c \in [0, c^{*})$, there exists $k_c \in \mathbb{N}$ with $k_{c}\geq N$ and vectors $\xi^1, \ldots, \xi^{k_c} \in \mathbb{Z}^N \setminus \{0\}$, such that, setting
\[Z_{-}^{j}:=\big\{\sum_{i=1}^{j-1}n_{i}\xi^{i}\mid n_{i}\in\mathbb Z\big\}\text{ and }Z_{+}^{j}=\mathbb Z^{N}\setminus Z_{-}^{j},\] 
for $j\in \{1,\ldots,k_{c}\}$ (with $Z_{-}^{1}=\{0\}$), the following hold:
\begin{itemize}
\item[(a)] $\xi^{j}\notin Z_{-}^{j}$ 
for any $j\in \{2,\ldots,k_{c}\}$,
\item[(b)] $\mathbb Z^{N}
=\{\sum_{i=1}^{k_{c}}n_{i}\xi^{i}\mid n_{i}\in\mathbb Z\}$,
\end{itemize}
and, for each $j \in \{1, \ldots, k_c\}$,  there exist an interval
$I_{j}=(\alpha_{j},\omega_{j})\subset\R$ with $-\infty\le\alpha_{j}<\omega_{j}\le+\infty$ and
a minimal $c$-connection $q_{c,j}:\bar I_{j}\to\R^{N}$ 
between $\mathcal V_c^-(Z_{-}^{j})$ and $\mathcal V_c^+(Z_{+}^{j})$,
such that, moreover,
\[
\lim_{t\to\alpha_{j}^+}\dist(q_{c,j}(t),\mathcal V_{c,0})=0\text{ and }
\lim_{t\to\omega_{j}^-}\dist(q_{c,j}(t),\mathcal V_{c,\xi^{j}})=0.\]
\end{proposition}
\begin{remark}
The connecting orbits $q_{c,j}$ for $j\in\{1,\ldots,k_c\}$ are geometrically
distinct. Indeed, let $j>i$. Since $\xi^{j}$ does not belong to the subgroup
generated by $\xi^{1},\ldots,\xi^{j-1}$, in particular we have
$\xi^{j}\neq \pm \xi^{i}$. 

Time translations preserve the endpoints of a connection, while time reflection
followed by a $\mathbb Z^{N}$--translation can only replace $\xi^{i}$ with
$-\xi^{i}$. Hence no combination of time translation, time reflection and
$\mathbb Z^{N}$--spatial translations preserving the initial point can transform $q_{c,i}$ into $q_{c,j}$.
\end{remark}

\begin{proof} Fix $c\in [0,c^{*})$ and proceed by induction. We start by considering the partition
\[Z_{-}^{1}=\{0\},\quad Z_{+}^{1}=\mathbb Z^{N}\setminus Z_{-}^{1}.\]
Since $Z_{-}^{1}$ plainly satisfies \eqref{eq:Z-}, Lemma \ref{R:2} gives the existence of $\xi^{1}\in\mathbb Z^{N}\setminus\{0\}$, an interval
$I_{1}=(\alpha_{1},\omega_{1})\subset\R$ with $-\infty\le\alpha_{1}<\omega_{1}\le+\infty$, and a minimal $c$--connection $q_{c,1}:\bar I_{1}\to\R^{N}$
between $\mathcal V_c^-(Z_{-}^{1})$ and $\mathcal V_c^+(Z_{+}^{1})$,
such that
\[
\lim_{t\to\alpha^+}\dist(q_{c,1}(t),\mathcal V_{c,0})=
\lim_{t\to\omega^-}\dist(q_{c,1}(t),\mathcal V_{c,\xi^{1}})=0\text{ and }\|q_{c,1}\|_{L^{\infty}(I_{1};\mathbb R^{N})}\leq 2R(c).\]
Assume that for some $j\ge 1$ there exist $\xi^1,\ldots,\xi^j\in\mathbb Z^N\setminus\{0\}$ such that, setting $Z_{-}^{i}=\{\sum_{\iota=1}^{i-1} n_\iota \xi^\iota \;:\; n_\iota\in\mathbb Z\}$ and $Z_{+}^{i}=\mathbb Z^{N}\setminus Z_{-}^{i}$, we have
\[
\xi^i\notin Z_{-}^{i}
\quad\text{for all } i=2,\ldots,j,
\qquad
\mathbb Z^N\neq S_j:=\{\sum_{\iota=1}^{j} n_\iota \xi^\iota \;:\; n_\iota\in\mathbb Z\},
\]
and that for each $i=1,\ldots,j$ there exist an interval $I_i=(\alpha_i,\omega_i)$ and a minimal $c$--connection
$q_{c,i}:\bar I_i\to\R^N$ between $\mathcal V_c^-(Z_{-}^{i})$ and $\mathcal V_c^+(Z_{+}^{i})$ such that
\[
\lim_{t\to\alpha_{i}^+}\dist(q_{c,i}(t),\mathcal V_{c,0})=
\lim_{t\to\omega_{i}^-}\dist(q_{c,i}(t),\mathcal V_{c,\xi^{i}})=0,
\qquad
\|q_{c,i}\|_{L^{\infty}(I_{i};\mathbb R^{N})}\le 2R(c).
\]

Define the new partition
\[
Z_-^{j+1}:=S_j,\qquad Z_+^{j+1}:=\mathbb Z^N\setminus Z_-^{j+1}.
\]
Since $S_j$ is a subgroup of $\mathbb Z^N$ containing $0$ and, by the inductive
assumption, $S_j\neq\mathbb Z^N$, the set $Z_-^{j+1}$ satisfies
\eqref{eq:Z-}. 
Let $\mathcal V_{c}^{-}(Z_-^{j+1})$ and $\mathcal V_{c}^{+}(Z_+^{j+1})$
be the corresponding decomposition of $\mathcal V_c$ given by
Lemma~\ref{lZ}. 
Lemma~\ref{R:2} then yields the existence of
$\xi^{j+1}\in Z_+^{j+1}$, an interval
$I_{j+1}=(\alpha_{j+1},\omega_{j+1})\subset\R$
with $-\infty\le\alpha_{j+1}<\omega_{j+1}\le+\infty$,
and a minimal $c$--connection $q_{c,j+1}$ between $\mathcal V_{c}^{-}(Z_-^{j+1})$ and $\mathcal V_{c}^{+}(Z_+^{j+1})$
such that
\[
\lim_{t\to\alpha_{j+1}^+}\dist(q_{c,j+1}(t),\mathcal V_{c,0})=
\lim_{t\to\omega_{j+1}^-}\dist(q_{c,j+1}(t),\mathcal V_{c,\xi^{j+1}})=0\text{ and }\|q_{c,j+1}\|_{L^{\infty}(I_{j+1};\mathbb R^{N})}\leq 2R(c).\]
By construction $\xi^{j+1}\in Z_+^{j+1}=\mathbb Z^N\setminus S_j$,
hence
\[
\xi^{j+1}\notin S_j=
\{\sum_{\iota=1}^{j} n_\iota \xi^\iota \;:\; n_\iota\in\mathbb Z\},
\]
so that $\xi^{j+1}$ does not belong to the subgroup generated by the previous generators.

By induction, as long as $S_j\neq\mathbb Z^N$, Lemma~\ref{R:2}
produces a new vector $\xi^{j+1}$ and a corresponding connecting orbit
$q_{c,j+1}$. 
Since $q_{c,j+1}$ connects $0$ to $\xi^{j+1}$ and
$\|q_{c,j+1}\|_{L^\infty(\R;\R^N)}\le 2R(c)$, we necessarily have
$|\xi^{j+1}|\le 2R(c)$. 
Hence only finitely many distinct integer vectors can occur.
Therefore there exists $k_c$ such that $S_{k_c}=\mathbb Z^N$,
and the construction stops. Since $S_{k_c}=\mathbb Z^N$ and $S_{k_c}$ is generated by
$\{\xi^1,\ldots,\xi^{k_c}\}$, we must have $k_c\ge N$.
\end{proof}


\end{document}